\documentclass[a4paper,12pt,leqno]{amsart}
\usepackage{amsmath}
\usepackage{amsfonts}
\usepackage{mathtools}
\usepackage{amssymb}
\usepackage{bbm}
\usepackage{amscd}
\usepackage{color}

\usepackage[colorlinks=true,urlcolor=blue,citecolor=blue,linkcolor=blue,linktocpage,pdfpagelabels,bookmarksnumbered,bookmarksopen]{hyperref}

\newtheorem{teo}{Theorem}[section]
\newtheorem{lem}[teo]{Lemma}
\newtheorem{cor}[teo]{Corollary}
\newtheorem{prop}[teo]{Proposition}

\theoremstyle{remark}
\newtheorem{oss}[teo]{Remark}
\theoremstyle{definition}
\newtheorem{defi}[teo]{Definition}

\newcommand{\bbR}{{\mathbb{R}}}

\newcommand{\bbN}{{\mathbb{N}}}
\newcommand{\bbQ}{{\mathbb{Q}}}
\newcommand{\bbZ}{{\mathbb{Z}}}

\newcommand{ \Rn} {{\mathbb {R}^n}}
\newcommand{ \Rm} {{\mathbb {R}^m}}

\newcommand{\bbH}{\mathbb{H}}

\newcommand{\eps}{\varepsilon}

\newcommand{\average}{{\mathchoice {\kern1ex\vcenter{\hrule height.4pt
width 6pt
depth0pt} \kern-9.7pt} {\kern1ex\vcenter{\hrule height.4pt width 4.3pt
depth0pt}
\kern-7pt} {} {} }}

\newcommand{\mA}{\mathcal{A}}

\newcommand{\ci}{\mathbf{C}}

\newcommand{\Lefh}{f_{h,e}}

\newcommand{\Czero} {{\hbox{\bf C}}_c }
\newcommand{\Co}[1]{{\bf C}^{#1}}

\renewcommand{\L}[1]{{L}^{#1}}

\newcommand{\W}[2]{{W}_X^{#1,#2} }
\newcommand{\HS}[2]{{H}_X^{#1,#2} }

\newcommand{\Lip}{{\hbox{\rm Lip}}}

\newcommand{\Om}{\Omega}

\newcommand{\norma}[1]{\Vert#1\Vert}

\newcommand{\scalare}[2]{\langle #1,#2 \rangle}

\newcommand{\vettore}[2]{(#1_1,\dots,#1_{#2})}

\makeatletter
\def\cleardoublepage{\clearpage\if@twoside \ifodd\c@page\else
\hbox{}
\thispagestyle{empty}
\newpage
\if@twocolumn\hbox{}\newpage\fi\fi\fi}
\makeatother
\title{$\Gamma$-convergence  for functionals depending on vector fields. II. Convergence of minimizers.}
\author{A. Maione}
\address{Alberto Maione: Abteilung f{\"u}r Angewandte Mathematik\\Albert-Ludwigs-Universit{\"a}t Freiburg\\Hermann-Herder-Stra{\ss}e 10\\79104 Freiburg i. Br. - Germany\\}
\email{alberto.maione@mathematik.uni-freiburg.de}
\thanks{A.M., A.P. and F.S.C. are supported by the Indam-GNAMPA project 2020 ``Convergenze variazionali per funzionali e operatori dipendenti da campi vettoriali\,''. A.M. is also supported by MIUR and the University of Freiburg, Germany. A.P. and F.S.C. are also supported by MIUR and the University of Trento, Italy.}
\date{\today}
\author{A. Pinamonti}
\address{Andrea Pinamonti: Dipartimento di Matematica\\Universit\`a di Trento\\ Via Sommarive 14\\ 38123, Povo (Trento) - Italy\\}
\email{andrea.pinamonti@unitn.it}
\author{F. Serra~Cassano}
\address{Francesco Serra Cassano: Dipartimento di Matematica\\Universit\`a di Trento\\ Via Sommarive 14\\ 38123, Povo (Trento) - Italy\\}
\email{francesco.serracassano@unitn.it}

\begin{document}


\begin{abstract} Given a family of locally Lipschitz  vector fields $X(x)=(X_1(x),\dots,X_m(x))$ on $\Rn$, $m\leq n$, we study integral functionals  depending on $X$. Using the results in \cite{MPSC1}, we study the convergence of minima, minimizers and  momenta of those functionals. Moreover, we apply these results to the periodic homogenization in Carnot groups and to prove a $H$-compactness theorem for linear differential operators of the second order depending on $X$.
\end{abstract}
\maketitle


\section{Introduction}


In this paper we deal with the asymptotic behaviour of minima, minimizers and momenta, as $h\to\infty$, of the following sequence of minimization problems
\begin{equation}\label{minpb}
\inf\left\{F_h(u)+G(u):\,u\in W^{1,p}_X(\Om),\,u-\varphi\in  W^{1,p}_{X,0}(\Om)\right\}.
\end{equation}
Here $F_h,\,G:L^p(\Om)\to\bbR\cup\{\infty\}$ denote the functionals 
\begin{align}\label{F_ph}
    F_h(u):=
    \displaystyle{\begin{cases}
\int_\Om f_h(x,Xu(x))dx&\text{ if }u\in W^{1,p}_X(\Om)\\
\infty&\text{ otherwise}
\end{cases}}
\end{align}
and
\begin{equation}\label{Gdatum}
G(u):=\int_\Om g(x,u(x))\,dx\,,
\end{equation}
with $f_h:\Om\times\Rm\to\bbR$, $h\in\mathbb{N}$, and $g:\Omega\times\mathbb{R}\to\mathbb{R}$ Carath\'eodory  functions and $X(x):=(X_1(x),\dots,X_m(x))$ a given family of first order linear differential operators with Lipschitz coefficients on a bounded open set $\Omega\subset\Rn$, that is,
\[
X_j(x)=\,\sum_{i=1}^nc_{ji}(x)\partial_i\quad j=1,\dots,m
\]
with $c_{ji}(x)\in Lip(\Om)$ for $j=1,\dots,m$, $i=1,\dots,n$. 

\noindent In the following, we will refer to $X$ and $f_h$ as {\it $X$-gradient} and {\it integrand function}, respectively. The environment spaces $W^{1,p}_X(\Om)$ and $W^{1,p}_{X,0}(\Om)$ are the Sobolev spaces associated through the $X$-gradient in a classical way,  according to Folland-Stein \cite{FS} (see Definitions \ref{Definition 1.1.1} and \ref{Definition 1.1.3}) and $\varphi\in W^{1,p}_X(\Om)$ is a given function that plays the role of boundary datum in \eqref{minpb}.
As usual, we identify each $X_j$ with the vector field $$(c_{j1}(x),\dots,c_{jn}(x))\in \Lip(\Omega,\Rn)$$
and we call 
\begin{equation}\label{coeffmatrvf}
C(x) = [c_{ji}(x)]_{
{i=1,\dots,n}\atop{j=1,\dots,
m}}
\end{equation}
the {\it coefficient matrix of the $X$-gradient}. 
 
Throughout the paper, we will assume some structural conditions on the class of integrand functions and the $X$-gradient. An integrand function $f:\,\Omega\times\Rm\to\bbR$ will typically satisfy the following conditions: 
\begin{itemize}
\item[($I_1$)] $f:\,\Om\times\Rm\to\bbR$ is Borel measurable on $\Om$;
\item[($I_2$)] for a.e. $x\in\Om$, the function $f(x, \cdot):\,\Rm\to\bbR$ is convex;
\item[($I_3$)] there exist two positive constants $c_0\leq\,c_1$ and two nonnegative functions $a_0,\,a_1\in L^1(\Omega)$ such that
\begin{equation}\label{3.2}
c_0\,|\eta|^p-a_0(x)\le\,f(x,\eta)\le\,c_1\,\left|\eta\right|^p+\,a_1(x)    
\end{equation}
for a.e. $x\in\Om$ and for each $\eta\in\Rm$.
\end{itemize}
We will denote  by $I_{m,p}(\Omega,c_0,c_1,a_0,a_1)$ the class of such integrand functions and by $I_{m,p}(\Omega,c_0,c_1)$ if $a_0=a_1\equiv 0$. Similarly, function $g$ in \eqref{Gdatum} will satisfy a suitable growth condition (see \eqref{crescitag} and \eqref{constg}). 

As far as the structural conditions on the $X$-gradient are concerned, we will need two assumptions: the former is an algebraic  condition and the latter a metric condition.

\begin{defi}\label{frkcond} We say that a family of vector fields $X(x)=\,(X_1(x),\break\dots, X_m(x))$ satisfies the {\it linear independence condition} (LIC) on an open set $\Omega\subset\Rn$ if there exists a set ${\mathcal N}_X\subset\Om$, closed in the topology of $\Om$, such that $|\mathcal N_X|=\,0$ and $X_1(x),\dots, X_m(x)$ are linearly independent as vectors of $\Rn$ for each $x\in\Om_X:=\,\Om\setminus\mathcal N_X$. Here $|A|$ denotes the $n$-dimensional Lebesgue measure of a measurable subset $A\subset\Rn$.
\end{defi}
\noindent Notice that, if $X$ satisfies (LIC) on $\Omega$, then $m\le\,n$.
In some results, we will also assume that  $X$ is defined and Lipschitz continuous on an open neighborhood $\Omega_0$ of $\overline{\Omega}$ and that the following conditions hold:
\begin{itemize}
\item[(H1)] Let $d:\mathbb{R}^n\times \mathbb{R}^n\to[0,\infty]$ be the so-called Carnot-Carath\'eodory distance function  induced by $X$ (see, for instance, \cite[Section 2]{FSSC2}). Then, $d(x,y)<\infty$ for any $x,y\in \Omega_0$, so that $d$ is a standard distance in $\Omega_0$, and $d$ is continuous with respect to the usual topology of $\mathbb{R}^n$.
\item[(H2)] For any compact set $K\subset\Omega_0$ there exist a radius $r_K$ and a positive constant $C_K$, depending on $K$, such that 
\[
|B_d(x,2r)|\leq C_K |B_d(x,r)|
\]
for any $x\in K$ and $r<r_K$.  $B_d(x,r)$ denotes the (open) metric ball with respect to $d$, that is, $B_d(x,r):=\{y\in\Omega_0\ |\ d(x,y)<r\}$. 
\item[(H3)] There exist geometric constants $c,C>0$ such that $\forall B=B_{d}(\overline{x},r)$ with $cB:=B_{d}(\overline{x}, cr)\subseteq \Omega_0$, $\forall u\in \mathrm{Lip}(c\overline{B})$ and $\forall x\in\overline{B}$
\[
\left |u(x)-\frac{1}{|B|}\int_B u(y) dy\right |\leq C\int_{cB}|Xu(y)| \frac{d(x,y)}{|B_d(x,d(x,y))|}\,dy\,.
\]
\end{itemize}
Let us point out that (LIC) embraces relevant and wide  families of  vector fields studied in literature (see \cite[Example 2.2]{MPSC1}), as well as important families of vector fields satisfying
conditions (H1), (H2) and (H3) (see Remark \ref{exvf}).

Each functional \eqref{F_ph}  always admits an integral representation with respect to the Euclidean gradient. Indeed, for instance, functional \eqref{F_ph} can be represented as follows
\[
F_h(u)=\,\int_\Om \Lefh (x,Du(x))\,dx\quad\text{for each }u\in\ci^1(\Om)\,,
\]
where $\Lefh:\,\Omega\times\Rn\to\bbR$ now denotes the {\it Euclidean integrand}, defined as
\begin{equation}\label{Li}
\Lefh(x,\xi):=\,f_h(x,C(x)\xi)\quad\text{ for a.e. }x\in\Om\,,\text{ for each }\xi\in\Rn.
\end{equation}
Notice also that we cannot reverse this representation (see  \cite[Counterexample 3.14]{MPSC1}) and the representation with respect to the Euclidean gradient could yield a loss of coercivity (see \cite{MPSC1}). Nonetheless, we will show that, by replacing the Euclidean gradient with the $X$-gradient, we can get rid of this drawback.

In this paper, we will exploit as  main tools for studying minimization problems \eqref{minpb} some results of $\Gamma$-convergence for functionals depending on vector fields.
In particular, let us recall a $\Gamma$-compactness theorem for the sequence $(F_h)_h$ defined in \eqref{F_ph} (see Theorem \ref{mainthm}), obtained in \cite{MPSC1}, to which we will refer for all the relevant definitions.
More precisely, if the $X$-gradient satisfies (LIC) condition and the sequence of integrand functions $(f_h)_h\subset I_{m,p}(\Omega,c_0,c_1,a_0,a_1)$, then, up to a subsequence not relabeled, we can assume the existence of a functional $F:\,L^p(\Om)\to\bbR\cup\{\infty\}$ and $f\in I_{m,p}(\Om,c_0,c_1,a_0,a_1)$ such that 
\begin{equation}\label{GammalimitFhintro}
F=\,\Gamma(L^p(\Om))\text{-}\lim_{h\to\infty} F_h
\end{equation}
and $F$ admits the following representation 
\begin{equation}\label{GammalimitFintro}
F(u):=
\displaystyle{\begin{cases}
\int_{\Om}f(x,Xu(x))dx&\text{ if }u\in W^{1,p}_X(\Om)\\
\infty&\text{ otherwise}
\end{cases}
\,.
}
\end{equation}

Let us now describe the main results of the present paper and some of their applications. First, recall the following Poincar\'e inequality on $W^{1,p}_{X,0}(\Omega)$, $1\le\,p<\infty$, which holds provided that  $\Omega$ is a bounded domain of $\bbR^n$ and $X$ satisfies conditions (H1), (H2) and (H3) (see Proposition \ref{PoincareW1p0}). More precisely, there exists a positive constant $c_{p,\Omega}$, depending only on $p$ and $\Om$, such that
\begin{equation}\label{PoincareW1p0for}
c_{p,\Omega}\,\int_\Om |u|^p\,dx\le\, 
\int_\Om |Xu|^p\,dx\quad\text{for each }u\in W^{1,p}_{X,0}(\Omega)\,.
\end{equation}
We will also assume that $c_{p,\Omega}$ is the best constant in the Poincar\'e inequality \eqref{PoincareW1p0for}, that is, it is the largest constant for which \eqref{PoincareW1p0for} holds.

Let us begin with a result concerning the convergence of minima and minimizers for minimization problems \eqref{minpb}.
For fixed $\varphi\in W^{1,p}_X(\Om)$, let $\mathbbm{1}_{\varphi}:L^p(\Om)\to\{0;\infty\}$ denote the indicator function of the affine subspace of  $W^{1,p}_{X}(\Omega)$
\[
W^{1,p}_{X,\varphi}(\Omega):=\,\varphi+\,W^{1,p}_{X,0}(\Omega)
\]
and assume that the Carath\'eodory function $g:\Om\times\bbR\to\bbR$ in \eqref{Gdatum} satisfies the following growth condition:
there exist two constants $d_0,d_1$ and two nonnegative functions $b_0,b_1\in L^1(\Omega)$ such that
\begin{align}\label{crescitag}
d_0|s|^p-b_0(x)\le g(x,s)\le d_1|s|^p+b_1(x) 
\end{align}
for a.e. $x\in\Omega$ and for every $s\in\bbR$, with
\begin{equation}\label{constg}
d_1>\,0\quad\text{and}\quad-c_0\,c_{p,\Omega}<\,d_0\leq d_1\,.
\end{equation}
\begin{teo}\label{convminimiz}
Let $\Omega$ be a bounded and connected open set, $1<p<\infty$ and let $X$ satisfy (LIC) condition and conditions (H1), (H2) and (H3). Let $f_h,f\in I_{m,p}(\Om,c_0,c_1,a_0,a_1)$, let $g$ satisfy \eqref{crescitag} and \eqref{constg} and let $F_h, G,F$ be the functionals in \eqref{F_ph}, \eqref{Gdatum} and \eqref{GammalimitFintro}, respectively.
Fixed $\varphi\in W_{X}^{1,p}(\Omega)$, let $\Xi^{\varphi}_h,\Xi^{\varphi}:L^p(\Omega)\to \bbR\cup\{\infty\}$ be, respectively, defined as
\begin{equation}\label{sumsum}
\Xi^{\varphi}_h:=F_h+G+\mathbbm{1}_{\varphi}\quad\text{and}\quad\Xi^{\varphi}:=F+G+\mathbbm{1}_{\varphi}\,.
\end{equation}
If $(F_h)_h$ $\Gamma$-converges to $F$ in the strong topology of $L^p(\Om)$, then
\begin{itemize}
    \item[$(i)$] for each $h\in\mathbb{N}$, both $\Xi^{\varphi}_h$ and $\Xi^{\varphi}$ attain their minima in $L^p(\Omega)$ and 
    \begin{equation}\label{convergenceminima}
        \min_{u\in L^p(\Omega)} \Xi^{\varphi}(u)=\lim_{h\to\infty}\min_{u\in L^p(\Omega)}\Xi^{\varphi}_h(u)\,;
    \end{equation}
    \item[$(ii)$] if $(u_h)_h$ is a sequence of minimizers of $(\Xi^{\varphi}_h)_h$, that is,
    \[
    \Xi^{\varphi}_h(u_h)= \min_{u\in L^p(\Omega)}\Xi^{\varphi}_h(u)\quad\text{for any } h\in\mathbb{N}\,,
    \]
    then, there exists $\bar u\in W_{X,\varphi}^{1,p}(\Omega)$ such that, up to subsequences, 
    \begin{equation}\label{uhtouLp}
    u_h\to\bar u\text{ weakly in }W_X^{1,p}(\Omega) \text{ and strongly in }L^p(\Om)
    \end{equation}
    and 
    \begin{equation}\label{barumin}
        \Xi^{\varphi}(\bar u)=\,\min_{u\in L^p(\Omega)}\Xi^{\varphi}(u)\,.
    \end{equation}
\end{itemize} 
\end{teo}
The second main result deals with the convergence of the momenta  associated with the  sequence of functionals  $(F_h)_h$ satisfying \eqref{GammalimitFhintro}. The result is inspired by \cite{ADMZ2} and it is a partial extension of those results to integral functionals depending on vector fields.
\begin{teo}\label{convmom} 
Let $\Om$ be a bounded open set, $1<p<\infty$ and let $X$ satisfy (LIC) condition. Let $f_h,f\in I_{m,p}(\Om,c_0,c_1,a_0,a_1)$, let $F_h,F$ be, respectively, the functionals in \eqref{F_ph} and \eqref{GammalimitFintro}, satisfying \eqref{GammalimitFhintro}, and define $\mathcal F_h,\mathcal{F}:\,L^p(\Om)^m\to\bbR$ as
\begin{align}\label{calFh}
    \mathcal F_h(\Phi):=\,\int_{\Om}f_h(x,\Phi(x))\,dx
\end{align}
and
\begin{align}\label{calF}
    \mathcal F(\Phi):=\,\int_{\Om}f(x,\Phi(x))\, dx
\end{align}
for any $\Phi\in L^p(\Om)^m$ and for any $h\in\mathbb{N}$.
Assume that:
\begin{itemize}
\item[$(i)$]
fixed $0\leq\alpha\leq\min\{1,p-1\}$, there exist a positive constant $\overline{c}$ and a nonnegative function $b\in L^p(\Om)$ such that
\begin{align*}
|\nabla_\eta f_h(x,\eta_1)-\nabla_\eta f_h(x,\eta_2)|\le\,\overline{c}\, |\eta_1-\eta_2|^{\alpha}\left(|\eta_1|+|\eta_2|+b(x)\right)^{p-1-\alpha}
\end{align*}
for a.e. $x\in\Om$, for any $\eta_1,\,\eta_2\in\Rm$ and for any $h\in\mathbb{N}$;
\item[$(ii)$]
the map $\Rm\ni\eta\mapsto f_h(x,\eta)$ belongs to $\ci^1(\mathbb{R}^m)$ for a.e. $x\in\Om$ and for any $h\in\mathbb{N}$;
\item[$(iii)$] the map $\Rm\ni\eta\mapsto f(x,\eta)$ belongs to $\ci^1(\Rm)$ for a.e. $x\in\Om$;
\item[$(iv)$] there exist $u_h,u\in W^{1,p}_X(\Om)$ such that
\[
u_h\to u\text{ in }L^p(\Om)\quad\text{and} \quad\mathcal F_h(Xu_h)\to \mathcal F(Xu)\quad\text{as }h\to\infty\,.
\]
\end{itemize}
Then, the convergence of momenta associated with $(F_h)_h$ holds, that is,
\begin{equation}\label{convmom0}
\partial_\Phi \mathcal F_h(Xu_h)=\,\nabla_\eta f_h(\cdot,Xu_h)\to\,\nabla_\eta f(\cdot,Xu)=\partial_\Phi\mathcal F(Xu)
\end{equation}
weakly in $L^{p'}(\Om)^m$, where $\partial_\Phi\mathcal{F}_h$ and $\partial_\Phi\mathcal F$ denotes, respectively, the Gateaux derivatives of functionals $\mathcal{F}_h$ and $\mathcal F$ (see \eqref{GderF}).
\end{teo}
As a consequence of Theorems \ref{convminimiz} and \ref{convmom}, we can infer the convergence of both minimizers and momenta associated with minimization problems \eqref{minpb}.
\begin{cor}\label{corconvmom} 
Let $\Omega$ be open, bounded and connected, let $1<p<\infty$ and let $X$ satisfy (LIC) condition and conditions (H1), (H2) and (H3). Let $f_h,f\in I_{m,p}(\Omega,c_0,c_1,a_0,a_1)$, let $g$ satisfy \eqref{crescitag} and \eqref{constg}, let $G$ be the functional \eqref{Gdatum} and let $F_h,\mathcal{F}_h,F,\mathcal{F}$ satisfy the hypotheses of Theorem \ref{convmom}. Fixed $\varphi\in W^{1,p}_{X}(\Omega)$, consider functionals $\Xi_h^{\varphi},\Xi^{\varphi}$ defined in \eqref{sumsum}.
If $(u_h)_h$ is a sequence of minimizers of $(\Xi_h^{\varphi})_h$ then, up to subsequences, there exists a minimizer $u$ of $\Xi^{\varphi}$ such that 
\[
u_h\to u\text{ weakly in }W^{1,p}_X(\Omega)\text{ and strongly in }L^p(\Om)\,.
\]
Moreover, \eqref{convmom0} holds.
\end{cor}
We will also  provide two interesting applications of the previous results  to the periodic homogenization of functionals in Carnot groups and the $H$-convergence  for linear differential operators of the second order depending on $X$. 

Let us  recall that $\Gamma$-convergence for functionals in \eqref{F_ph} has been studied in the framework of {Dirichlet forms} \cite{Fu,MR}, but for special  integrand functions $f$ and $X$-gradient satisfying the H\"ormander condition (see, for instance, \cite{BPT2,BT,Mo} and references therein). We also point out that $\Gamma$-convergence for functionals defined in Cheeger-Sobolev metric measure spaces  has been also studied (see \cite{AHM} and references therein). 

Homogenization in Carnot groups has been intensively studied so far (see for instance, \cite{BMT,BPT1,BPT2,FGVN,FT,MV}). Here we are interested in the recent paper \cite{DDMM}, where a $\Gamma$-convergence result for the periodic homogenization of functionals in Heisenberg groups has been proved  (see Theorem \ref{resultDDMM}). By using this result, we prove the convergence  of minimizers for minimization problems \eqref{minpb}, for each boundary datum $\varphi\in W^{1,p}_X(\Om)$, as well as the convergence of the associated momenta (see Corollary \ref{corconvmomhomog}). 

The $H$-convergence for subelliptic PDEs has been also studied in the setting of Carnot groups (see \cite{BFT,BFTT,FTT,Ma2}). In the previous papers, the main tool for showing a compactness result for $H$-convergence is the nontrivial extension to Carnot groups of the so-called {\it compensated compactness} \cite{BFTT}, originally introduced in the classical Euclidean setting by Murat and Tartar \cite{Mu}. Here we get a compactness result for $H$-convergence in a broader setting than a Carnot group (see Theorem \ref{FDC}), by using Theorem \ref{convmom}, without applying the compensated compactness, since it is not clear whether it still holds in our framework.

The plan of the paper is as follows: in section \ref{sec2}, we introduce and study  the Sobolev spaces associated with the $X$-gradient. In section \ref{sec3}, we prove a criterion for the $\Gamma$-convergence with respect to the weak convergence in spaces  $W^{1,p}_{X}(\Om)$  and  $W^{1,p}_{X,0}(\Om)$ (see Theorem \ref{GammaWeak}). We also recall two key  results of $\Gamma$-compactness for integral functions depending on $X$, with respect to $L^p(\Omega)$-topology: the former is a light extension of a result already showed in \cite{MPSC1} (see Theorem \ref{mainthm}) and the latter is  a $\Gamma$-compactness result  including Dirichlet boundary conditions (see Theorem \ref{convboundary}). In section \ref{sec4}, we prove Theorems \ref{convminimiz} and \ref{convmom} and Corollary \ref{corconvmom}. Finally, in sections \ref{sec5} and \ref{sec6}, we apply Theorems \ref{convminimiz} and \ref{convmom} and Corollary \ref{corconvmom} to the case of periodic homogenization in  Heisenberg groups (see Corollary \ref{corconvmomhomog}) and to the $H$-convergence 
for linear differential operators of the second order depending on $X$ (see Theorem \ref{FDC}), respectively.
\vskip5pt

\noindent{\it Acknowledgements.} We thank G. Buttazzo and B. Franchi for useful suggestions and discussions on these topics. We also thank the referee of the paper \cite{MPSC1} and M. Caponi for  useful suggestions on the proof of Theorem \ref{GammaWeak}.


\section{Functional setting}\label{sec2}


Throughout the paper, $\Omega\subset\Rn$ is a fixed open set and $\overline{\mathbb{R}}=[-\infty,\infty]$. If $v, w\in\Rn$, we denote
by $|v|$ and $\scalare v w$ the Euclidean norm and the scalar product,
respectively. 
If $\Om$ and $\Om'$ are subsets of $\Rn$, then $\Om'\Subset \Om$ means that $\Om'$ is compactly contained in $\Om$. Moreover, $B(x,r)$ is the open Euclidean ball of radius $r$ centered at $x$. 
If $A\subset \Rn$, $\chi_A$ and $\mathbbm{1}_A$ are, respectively, the characteristic and the indicator function of $A$, $|A|$ is its $n$-dimensional Lebesgue measure $\mathcal L^n$ and, by notation {\it a.e. $x\in A$}, we will simply mean $\mathcal L^n$-a.e. $x\in A$. 
In the sequel, we denote by $\Co k (\Omega)$ the space of $\bbR$-valued functions $k$ times
continuously differentiable and by $\Co k_c (\Omega)$  the subspace of $\Co k (\Omega)$ whose functions have support compactly contained in $\Omega$. 
\begin{defi}
For any $u\in\L 1(\Om)$ we define $Xu$ 
as an element of ${\mathcal D}'(\Om;\Rm)$ as follows
\[
\begin{split}
Xu(\psi):&=(X_1u(\psi_1),\dots,X_mu(\psi_m))\\
         &=-\int_\Om u\left(\sum_{i=1}^n\partial_i( c_{1i}\,\psi_1),\dots, \sum_{i=1}^n\partial_i( c_{mi}\,\psi_m)\right)dx
\end{split}
\]
for any $\psi=\vettore {\psi} m \in\Czero^\infty(\Om;\Rm)$.
\end{defi}
If we set $X^T\psi:=(X^T_1\psi_1,\dots, X^T_m\psi_m)$, with
\begin{equation}\label{XjT}
X_j^T\varphi:=\,-\sum_{i=1}^n\partial_i( c_{ji}\,\varphi)=\,-\left(\mathrm {div}(X_j)+X_j\right)\varphi
\end{equation}
for any $\varphi\in\ci^\infty_c(\Om)$ and $j=1,\dots,m$, then 
the aspect of the previous definition is even more familiar
\[
Xu(\psi)=\,\int_\Om u\, X^T\psi\,dx\quad\text{for any }\psi \in\Czero^\infty(\Om;\Rm)\,.
\]
\begin{oss}
By the well-known extension result for Lipschitz functions, without loss of generality, we can assume that vector fields' coefficients $c_{ji}\in Lip_{\rm loc}(\Rn)$ for any $j=1,\dots,m$ and $i=1,\dots,n$.
\end{oss}
\begin{defi}\label{Definition 1.1.1}
For $1\leq p\leq\infty$ we set
\[
\begin{split}
{W}_X^{1, p}(\Om)&:=\left\{ u\in L^ p(\Om):X_j u\in L^ p(\Om)\ {\hbox{\rm
for }} j=1,\dots,m\right\}\\
{W}_{X;loc}^{1, p}(\Om)&:=\left\{u: u|_{\Om'}\in W_X^{ 1,p}(\Om')\text{ for every open set }\Om'\Subset\Om\right\}.
\end{split}
\]
\end{defi}
\begin{oss} Since vector fields $X_j$ have  locally Lipschitz continuous coefficients, $\partial_{i }c_{ji}\in L_{\rm loc}^\infty(\Rn)$ for any $j=1,\dots,m$ and $i=1,\dots,n$.
Then, by definition,
\begin{equation}\label{inclclassSobsp}
W^{1,p}(\Om)\subset\W1 p(\Om)\quad\forall\,p\in [1,\infty]
\end{equation}
for any open bounded set $\Omega\subset\Rn$. Moreover, for any $u\in W^{1,p}(\Omega)$
\begin{equation*}
Xu(x)=\,C(x)\,Du(x)\quad\text{ for a.e. }x\in\Om\,,
\end{equation*}
where $W^{1,p}(\Om)$ denotes the classical Sobolev space, or, equivalently, the space $\W1 p(\Om)$ associated to
\[
X=\,D:=\,(\partial_1,\dots,\partial_n)\,.
\]
It is easy to see that inclusion \eqref{inclclassSobsp} can be strict and turns  out to be continuous. As well, there is the inclusion
\begin{equation*}
W^{1,p}_{\rm loc}(\Om)\subset W_{X;{\rm loc}} ^{1,p}(\Om)\quad\forall\,p\in [1,\infty]\,.
\end{equation*}
\end{oss}
The following proposition is proved in \cite{FS} and \cite[Lemma 2.3.29]{Ma}.
\begin{prop}\label{psaw} $\W 1 p (\Om)$ endowed with the norm
\[
\|u\|_{W^{1,p}_X(\Omega)}:=\left(\int_{\Omega} |u|^p\, dx + \int_{\Omega} |Xu|^p\, dx\right)^\frac{1}{p}
\]
is a reflexive Banach space if $1<p<\infty$.
\medskip

\noindent Moreover, if $p>1$, then functional $\|\cdot\|^p_{W^{1,p}_X(\Omega)}:W^{1,p}_X(\Omega)\to[0,\infty)$ is lower semicontinuous and sequentially coercive in the weak topology of $W^{1,p}_X(\Omega)$.
\end{prop}
\begin{defi}\label{Definition 1.1.3} For $1\leq p\leq\infty$ we set
\[
\HS 1 p (\Om)\ {\hbox{\rm the closure of }} \Co 1 (\Om)\cap {\W 1 p (\Om)} \;
{\hbox{\rm in  }} {\W 1 p (\Om)}\,,
\]
\[
W^{1,p}_{X,0}(\Om)\ {\hbox{\rm the closure of }} \Co 1_c (\Om)\cap {\W 1 p (\Om)} \;
{\hbox{\rm in  }} {\W 1 p (\Om)}\,.
\]
\end{defi}
It is proved in \cite{FS} that the normed spaces $(H_{X}^{1,p}(\Omega),\|\cdot\|_{\W 1p(\Omega)})$ and $(W^{1,p}_{X,0}(\Omega),\|\cdot\|_{\W 1p(\Omega)})$ are Banach spaces for any $1\leq p \leq \infty$.

As for the usual Sobolev spaces, it holds that $ \HS 1 p (\Om)\subset {\W 1 p (\Om)}$. The classical result \lq$H=W$\rq\, of Meyers and Serrin \cite{MS} still holds true for these anisotropic Sobolev spaces as proved, independently, in \cite{FSSC1} and \cite{GN}.
Analogous results, under some additional assumptions, are proved in \cite{FSSC2}, for the weighted case, and in \cite{APS}, where a generalization to metric measure spaces is given.
\begin{teo}\label{Theorem 1.2.3} Let $\Om$ be an open subset of $\Rn$ and $1\leq p <\infty$. Then, 
\[
\HS 1 p (\Om)= {\W 1 p (\Om)}\,.
\]
\end{teo}
We conclude this section recalling that, when $\Omega$ is bounded and the family $X$ satisfies properties (H1), (H2) and (H3), then $W^{1,p}_{X,0}(\Omega)$ can be compactly embedded in $L^p(\Omega)$ for any $1\leq p<\infty$.
Moreover, we prove that, if in addition $\Omega$ is connected, then
\begin{equation*}
\|u\|_{W^{1,p}_{X,0}(\Omega)}:=\left(\int_{\Omega} |Xu|^p\, dx\right)^{\frac{1}{p}}
\end{equation*}
defines an equivalent norm on $W^{1,p}_{X,0}(\Omega)$ for any $1\leq p<\infty$. 
\medskip

Let us point out some classes  of relevant vector fields satisfying properties (H1), (H2) and (H3).
\begin{oss}\label{exvf}
{\bf(i)} ({\it H\"ormander vector fields}) If the vector fields are smooth and the rank of the Lie algebra generated by $X_1,\ldots, X_m$ equals $n$ at any point of $\Omega_0$ (the so-called H\"ormander condition), then (H1), (H2) and (H3) hold (see \cite{NSW} for (H1) and (H2) and \cite{FLW} for (H3)).
\medskip

\noindent {\bf(ii)} ({\it Grushin vector fields}) If the vector fields are as in \cite{F2}, \cite{F1} and \cite{FL}, then conditions (H1), (H2) and (H3) still hold (see \cite{F2, F1, FL} for (H1) and (H2) and \cite[Remark 2.8]{FSSC2} for (H3)).
\end{oss}
The following results are proved in \cite[Theorems 2.11 and 3.4]{FSSC2}.
\begin{teo}\label{Poincare}
Let $\Omega\subset\mathbb{R}^n$ be a bounded open set, let $1\leq p<\infty$ and let $X$ satisfy conditions (H1), (H2) and (H3). Then, for each metric ball $B=B_d(x,r)\subset\Omega$ and for every $u\in W^{1,p}_{X}(\Omega)$, there exist a constant $c$, depending on $B$ and $u$, and a constant $C$, not depending on $u$, such that
\[
\int_B|u(x)-c|^p\,dx\le\,C\,r^p\,\int_B|Xu|^p\,dx\,.
\]
\end{teo} 
\begin{teo}\label{immersion}
Let $\Omega\subset\mathbb{R}^n$ be a bounded open set, let $1\leq p<\infty$ and let $X$ satisfy conditions (H1), (H2) and (H3). Then, $W^{1,p}_{X,0}(\Omega)$ is compactly embedded in $L^p(\Omega)$.
\end{teo} 
An interesting consequence of Theorem \ref{immersion} is the following result.
\begin{prop}\label{compemb}
Under the assumptions of Theorem \ref{immersion}, $W^{1,p}_{X}(\Omega)$ can be compactly embedded in $L^p_{loc}(\Omega)$.
\end{prop}
\begin{proof}
Since $\Omega$ is open, there exists a sequence of open subsets of $\Omega$, $\emptyset\neq \Omega_1\subseteq\overline{\Omega}_1\subseteq \Omega_2\subseteq\dots$, such that $\overline{\Omega}_i$ is compact for every $i\in \mathbb{N}$ and 
\[\bigcup_{i=1}^\infty\,\overline{\Omega}_i=\Omega\,.
\]
Let $(u_n)_{n}$ be a bounded sequence in $W^{1,p}_{X}(\Omega)$, let $\varphi\in C^{\infty}_c(\Omega)$, with $0\leq \varphi\leq 1$ in $\Omega$ and $\varphi\equiv 1$ on $\Omega_1$, and define
\[
v_n^{(1)}:=\varphi\, u_n|_{\Omega_1}\quad\text{for any }n\in\mathbb{N}\,.
\]
Then, $(v_n^{(1)})_n$ is bounded in $W^{1,p}_{X,0}(\Omega)$ and, by Theorem \ref{immersion}, there exist a subsequence $(v_{n_k}^{(1)})_{k}$ of $(v_n^{(1)})_{n}$ and $u^{(1)}\in L^p(\Omega)$ such that
\begin{align*}
    v_{n_k}^{(1)}\rightarrow u^{(1)}\quad\text{in } L^p(\Omega)\quad\text{and}\quad u_n^{(1)}\rightarrow u^{(1)}\quad\text{in } L^p(\Omega_1)\,,
\end{align*}
where $(u_n^{(1)})_{n}$ is the subsequence of $(u_n)_{n}$ such that $v_{n_k}^{(1)}=\varphi\, u_n^{(1)}|_{\Omega_1}$.

Repeating the procedure described above for $(u_n^{(1)})_{n}$, we find the existence of $u^{(2)}\in L^p(\Omega)$ and $(u_n^{(2)})_{n}$, subsequence of $(u_n^{(1)})_{n}$, such that 
\begin{align*}
    u_n^{(2)}\rightarrow u^{(2)}\quad\text{in } L^p(\Omega_2)\quad\text{and}\quad u^{(1)}=u^{(2)}\quad\text{a.e. in } \Omega_1\,.
\end{align*}
Let us iterate the procedure for any $i\in\mathbb{N}$. Define
\[
v_n:=u_n^{(n)}\quad\text{ for any }n\in\mathbb{N}
\]
the $n$-th element of the $n$-th subsequence of $(u_n)_{n}$ and
\[
\overline{u}(x):=u^{(i)}(x)\quad\text{if }x\in\Omega_i\,.
\]
Then, by construction, $\overline{u}$ is well-defined and $\overline{u}\in L^p(\Omega)$. To conclude the proof, let us show that
\begin{align*}
    v_n\rightarrow\overline{u}\quad\text{in }L^p(\tilde{\Omega})\quad\text{for any open set }\tilde{\Omega}\Subset\Omega.
\end{align*}
Fixed $\tilde{\Omega}\Subset\Omega$, there exists $i\in\mathbb{N}$ such that $\tilde{\Omega}\Subset\Omega_i$ and so
\begin{align*}
    \int_{\tilde{\Omega}}|v_n-\overline{u}|^p\,dx\leq\int_{\Omega_i}|v_n-\overline{u}|^p\,dx=\int_{\Omega_i}|u_n^{(n)}-u^{(i)}|^p\,dx\to 0
\end{align*}
as $n\to\infty$. Then, the conclusion trivially follows.
\end{proof}
Let us point out that we can get a global compact embedding in Proposition \ref{compemb}, by requiring further 
regularity on $\Omega$, in the sense of the following definition.
\begin{defi} 
Let $(M,d)$ be a metric space. A bounded set $\Omega\subset\,M$ is said to be a {\it uniform domain} if there exists $\eps>\,0$ such that for each $x,y\in\Om$ there exists a continuous rectifiable curve $\gamma:\,[0,1]\to\Om$, with
\[
\gamma(0)=\,x,\quad\gamma(1)=\,y\,,
\]
\[
{\rm length}(\gamma)\le\,\frac{1}{\eps}\,d(x,y)\,,
\]
and, for each $t\in [0,1]$,
\[
{\rm dist}(\gamma(t),\partial\Om)\ge\, \eps\,\min\{{\rm length}(\gamma|_{[0,t]}), {\rm length}(\gamma|_{[t,1]})\}\,.
\]
\end{defi}
\begin{oss} The characterization of uniform domains in metric spaces is  a difficult task. Few examples of such domains are known, also in the framework of the Carnot-Carath\'eodory distance. A comprehensive account of uniform domains, with respect to the Carnot-Carath\'eodory distance, can be found in \cite{Mon} (see also \cite{FPS} for an interesting example).
\end{oss}
If $\Omega$ is a (bounded) uniform domain in the metric space $(\Omega_0,d)$, then, by using an extension result for functions in $W^{1,p}_{X}(\Omega)$ (see \cite{GN2}), by applying Theorem \ref{immersion} and by a localization argument in a neighborhood of $\Omega$, we get the following result.
\begin{teo}\label{immersionnot0} Let $\Omega\subset\mathbb{R}^n$ be a bounded open set, let $1\leq p<\infty$ and let $X$ satisfy conditions (H1), (H2) and (H3). Moreover, assume that $\Omega$ is a uniform domain in the metric space $(\Omega_0,d)$. Then, $W^{1,p}_{X}(\Omega)$ can be compactly embedded in $L^p(\Omega)$.
\end{teo}
\begin{oss} Actually, Theorem \ref{immersionnot0} still holds for an even more general class of metric regular sets, namely the so-called {\it PS-domains} (see \cite[Theorem 1.28]{GN}). In particular, the metric balls with respect to the Carnot-Carath\'eodory distance are PS domains (see \cite{FGW,GN}).
\end{oss}
As a consequence of Theorems \ref{Poincare} and \ref{immersion}, a global Poincar\'e inequality holds in $W^{1,p}_{X,0}(\Omega)$.
\begin{prop}\label{PoincareW1p0}
Let $\Omega\subset\mathbb{R}^n$ be a bounded open set, let $1\leq p<\infty$ and let $X$ satisfy conditions (H1), (H2) and (H3). Moreover, assume that $\Omega$ is connected. Then, there exists a positive constant $c_{p,\Omega}$, depending on $p$ and $\Omega$, such that 
\begin{equation*}
c_{p,\Omega}\,\int_\Om |u|^p\,dx\le\, 
\int_\Om |Xu|^p\,dx\quad\text{for any }u\in W^{1,p}_{X,0}(\Omega)\,.
\end{equation*}
\end{prop}
\begin{proof}
Let $\Omega_1\subset\mathbb{R}^n$ be a bounded and connected open set such that $\Omega\Subset\Omega_1\Subset\Omega_0$.
Let us first show that, for any $u\in W^{1,p}_{X,0}(\Omega)$, its extension $\bar u$ to $\Omega_1$, defined as 
\[
\bar u\equiv 0\quad\text{in }\Omega_1\setminus\Omega\,,
\]
that is,
\begin{equation}\label{extu}
X\bar u=
\begin{cases}
Xu&\text{ in }\Omega\\
0&\text{ in }\Omega_1\setminus\Omega
\end{cases}
\quad\text{a.e. in }\Omega_1\,,
\end{equation}
satisfies $\bar u\in W^{1,p}_{X,0}(\Omega_1)$.

\noindent Let $(u_h)_h\subset\ci^1_c(\Omega)$ be such that $u_h\to u$ in $W^{1,p}_{X}(\Omega)$. Then, it is easy to see that $(\bar u_h)_h\subset\ci^1_c(\Omega_1)$ and $\bar u_h\to \bar u$ in $W^{1,p}_{X}(\Omega_1)$.

Assume now, by contradiction, the existence of $(v_h)_h\subset W^{1,p}_{X,0}(\Omega)$ such that
\begin{equation*}%
\int_\Om |v_h|^p\,dx>\, h
\int_\Om |Xv_h|^p\,dx\quad\text{for each }h\in\bbN\,.
\end{equation*}
By \eqref{extu}, it follows that
\begin{equation}\label{PoincareW1p01}
\int_{\Om_1} |\bar v_h|^p\,dx>\, h
\int_{\Om_1} |X\bar v_h|^p\,dx\quad\text{for each }h\in\bbN\,.
\end{equation}
Let
\[
w_h(x):=\,\frac{\bar v_h(x)}{\|\bar v_h\|_{L^p(\Om_1)}}\quad\text{if }x\in\Om_1\,.
\]
Then, by construction, $(w_h)_h\subset W^{1,p}_{X,0}(\Omega_1)$ and
\begin{equation}\label{PoincareW1p04}
\int_{\Om_1} |w_h|^p\,dx=1\quad\text{for each }h\ge\,1\,.
\end{equation}
Moreover, by \eqref{PoincareW1p01}, 
\begin{equation}\label{PoincareW1p05}
\int_{\Om_1} |Xw_h|^p\,dx<\,\frac{1}{h}\quad\text{for each }h\ge\,1\,.
\end{equation}
By \eqref{PoincareW1p04}, \eqref{PoincareW1p05} and, in virtue of Theorem \ref{immersion}, there exists $w\in W^{1,p}_{X,0}(\Omega_1)$ such that, up to subsequences,
\begin{equation}\label{PoincareW1p06}
w_h\to w\text{ in }L^p(\Omega_1)\text{ and a.e. in }\Omega_1\,,
\end{equation}
\begin{equation}\label{PoincareW1p061}
\int_{\Omega_1}|w|^p\,dx=\,1
\end{equation}
and
\begin{equation}\label{PoincareW1p07}
Xw=\,(0,\dots,0)\quad\text{a.e. in }\Omega_1\,.
\end{equation}
Then, by \eqref{PoincareW1p07} and, in virtue of Theorem \ref{Poincare}, $w$ is locally constant on $\Omega_1$ and, since $\Om_1$ is connected, there exists a constant $k\in\bbR$ such that 
\[
w(x)=\,k\quad\text{for a.e. }x\in\Omega_1\,.
\]
By \eqref{PoincareW1p061}, $k\neq0$ and this yields a contradiction since, by definition, $w_h\equiv 0$ in $\Omega_1\setminus\bar\Omega$ for any $h\in\mathbb{N}$ and, by \eqref{PoincareW1p06}, $w=k=0$ a.e. in $\Omega_1\setminus\bar\Omega$.
\end{proof}
\begin{cor} Let $\Omega$, $p$ and $X$ be as in Proposition \ref{PoincareW1p0}. Then, 
\begin{equation*}
\|u\|_{W^{1,p}_{X,0}(\Omega)}:=\left(\int_{\Omega} |Xu|^p\, dx\right)^{\frac{1}{p}}
\end{equation*}
is a norm in $W^{1,p}_{X,0}(\Omega)$ equivalent to $\norma{\cdot}_{W^{1,p}_X(\Omega)}$.
\end{cor}
We conclude this section recalling the following estimate that will be useful to prove the coercivity of functionals $\Xi^{\varphi}_h$, defined in \eqref{sumsum}. It can be proved as in \cite[Lemma 2.7]{DM}.
\begin{lem}\label{PoincareW1pvarphi} Let $c_{p,\Om}$ be the Poincar\'e constant in \eqref{PoincareW1p0for}, let $\varphi\in W^{1,p}_{X}(\Omega)$ and let $c< c_{p,\Om}$. Then, there exist a positive constant $k_1$, depending only on $c$ and $c_{p,\Om}$, and a nonnegative constant $k_2$, depending on $c$, $c_{p,\Om}$ and $\|\varphi\|_{W^{1,p}_{X}(\Omega)}$, such that
\begin{equation*}
\int_\Om |Xu|^p\,dx-c\,\int_\Om |u|^p\,dx\ge\,k_1\left( \int_\Om |Xu|^p\,dx+ \int_\Om |u|^p\,dx\right)-k_2
\end{equation*}
for every $u\in W^{1,p}_{X,\varphi}(\Omega)$.
\end{lem}


\section{\texorpdfstring{$\Gamma$}{TEXT}-convergence results for integral functionals depending on vector fields}\label{sec3}


In this section, we study $\Gamma$-convergence results for classes of integral functionals depending on vector fields, with respect to the weak topology of $W^{1,p}_{X,0}(\Omega)$ and $W^{1,p}_{X}(\Omega)$, namely Theorem \ref{GammaWeak}, and the strong topology of $L^p(\Omega)$, see Theorems \ref{mainthm}, \ref{convboundary} and \ref{pertfunct}.

\subsection{\texorpdfstring{$\Gamma$}{TEXT}-convergence in the weak topology of \texorpdfstring{$W^{1,p}_{X,0}(\Omega)$}{} and \texorpdfstring{$W^{1,p}_X(\Omega)$}{}}

First we show that, if $X$ satisfies conditions (H1), (H2) and (H3), then the pointwise convergence of the sequence $(f_h(\cdot,\eta))_h$ a.e. in $\Omega$ for any $\eta\in\mathbb{R}^m$ implies the $\Gamma$-convergence of the corresponding integral functionals in the weak topology of $W^{1,p}_{X,0}(\Om)$ and $W^{1,p}_{X}(\Om)$.
\begin{teo}\label{GammaWeak} 
Let $\Omega\subset\mathbb{R}^n$ be a bounded open set, $1<p<\infty$ and let $X$ satisfy (LIC) condition and conditions (H1), (H2) and (H3). Let $f_h,f\in I_{m,p}(\Om,c_0,c_1,a_0,a_1)$, with $a_i\in L^\infty(\Om)$ ($i=\,0,\,1$), and let $F_h,\,F:\,W^{1,p}_{X,0}(\Om)\to\bbR$ be the corresponding integral functionals, defined as
\begin{equation}\label{FhFweak}
    F_h(u):=\,\int_\Om f_h(x,Xu(x))\,dx\,,\quad
F(u):=\,\int_\Om f(x,Xu(x))\,dx
\end{equation}
for any $u\in W^{1,p}_{X,0}(\Om)$ and for any $h\in\mathbb{N}$. Assume that
\begin{equation*}
f_h(\cdot,\eta)\to f(\cdot,\eta)\quad\text{a.e. in }\Omega\text{ and for any }\eta\in\Rm.
\end{equation*}
Then, $(F_h)_h$ $\Gamma$-converges
to $F$ in the weak topology of $W^{1,p}_{X,0}(\Om)$.

Moreover, if $W^{1,p}_X(\Om)$ is compactly embedded in $L^p(\Om)$, then $(F_h+G)_h$ $\Gamma$-converges to $F+G$ in the weak topology of $W^{1,p}_X(\Om)$. In this case, functionals $F_h,\,F:\,W^{1,p}_X(\Om)\to\bbR$ are defined as
\begin{equation*}
    F_h(u):=\,\int_\Om f_h(x,Xu(x))\,dx\,,\quad
F(u):=\,\int_\Om f(x,Xu(x))\,dx
\end{equation*}
for any $u\in W^{1,p}_X(\Om)$ and for any $h\in\mathbb{N}$, while $G:W^{1,p}_X(\Om)\to\mathbb{R}$ is the functional in \eqref{Gdatum} such that $g$ satisfies \eqref{crescitag}, with $0<\,d_0\leq d_1$.
\end{teo}
Before proving Theorem \ref{GammaWeak} we need two technical lemmas.
\begin{lem}\label{const}
Let $f\in I_{m,p}(\Omega,c_0,c_1,0,a)$, with $a\in L^\infty(\Om)$, and let $r>0$. There exist $R=R(r)>r$ and a Borel measurable function $g_{r}:\Omega\times\mathbb{R}^m\to\mathbb{R}$ such that $g_r(x,\cdot)$ is convex for a.e. $x\in\Omega$ and 
\begin{align}\label{condG2}
&0\leq g_{r}(x,\eta)\leq\,f(x,\eta)\quad \mbox{for a.e. $x\in\Omega$ and every $\eta\in \mathbb{R}^m$}\,;\\
\label{condG}
&g_{r}(x,\eta)=\,f(x,\eta)\quad \mbox{for a.e. $x\in\Omega$ and every $\eta\in\overline{B_{r}(0)}$}\,;\\
\label{condG3}
&g_{r}(x,\eta)\leq c_0|\eta|^p\quad\mbox{for a.e. $x\in\Omega$ and every $\eta\in\mathbb{R}^m\setminus\overline{B_R(0)}$}\,.
\end{align}
\end{lem}
\begin{proof} Observe that, without loss of generality, we can assume that
\begin{equation}\label{I2everywhere}
f(x,\cdot):\,\Rm\to [0,\infty)\text{ is convex for each }x\in\Om\,.
\end{equation}
Indeed, by $(I_2)$ and since the $n$-dimensional Lebesgue measure $\mathcal{L}^n$ is Borel regular, there exists a negligible Borel set $N\subset\Om$ such that $f(x,\cdot):\,\Rm\to[0,\infty)$ is convex for each $x\in\Om\setminus N$.
By redefining $f(x,\eta):=\,0$ for each $x\in N$ and $\eta\in\Rm$, we get the desired conclusion.

By \eqref{I2everywhere} and \cite[Theorem 10.4]{Ro}, $f(x,\cdot):\,\Rm\to[0,\infty)$ is locally Lipschitz for any $x\in\Om$. In particular, fixed $x\in\Om$, $\eta_0\in\mathbb{R}^m$ and $r>0$, there exists a positive constant $L$, depending on $f,x,\eta_0$ and $r$, such that
\begin{equation}\label{fxlip}
|f(x,\eta_1)-f(x,\eta_2)|\le\,L\,|\eta_1-\eta_2|
\end{equation}
for any $\eta_1,\,\eta_2\in \overline{B_r(\eta_0)}$, and 
\[
L:=\frac{1}{r}\sup_{\overline{B_{2r}(\eta_0)}}f(x,\cdot)\,.
\]
Fix $x\in\Omega$ and $\eta_0\in\mathbb{R}^m$. By \cite[Proposition 2.22]{Cla} and \eqref{fxlip}, for any $v\in\Rm$ there exists the directional derivative 
\[
f_x'(\eta_0,v):=\,\lim_{t\to 0}\frac{f(x,\eta_0+tv)-f(x,\eta_0)}{t}\in\bbR\,. 
\]
It is also clear that
\[
f'_x(\eta_0,v)=\lim_{h\to \infty}h\left(f(x,\eta_0+\frac{1}{h}v)-f(x,\eta_0)\right).
\]
By \eqref{fxlip} and \cite[Corollary 4.26]{Cla},
\begin{equation}\label{f'max}
f'_x(\eta_0,v)=\max_{\xi\in\partial f_x(\eta_0)} \langle \xi, v\rangle \quad\text{for any }v\in \mathbb{R}^m,
\end{equation}
where $\partial f_x(\eta_0)$ denotes the subdifferential of $f(x,\cdot)$ at $\eta_0$.

\noindent By \eqref{f'max}, the map $\mathbb{R}^m\ni v\mapsto f'_x(\eta,v)$ is positively homogeneous of degree one, convex, and so subadditive, continuous and finite. Moreover, since $f:\,\Omega\times\Rm\to[0,\infty)$ is Borel measurable, then the map $\Omega\times\Rm\ni (x,v)\mapsto f_x'(\eta,v)$ is Borel measurable for any $\eta\in\mathbb{R}^m$.

Fix $r>0$ and $\eta_0,\eta\in \mathbb{R}^m$. We define
\begin{equation*}
G_{\eta_0}(x,\eta):=f(x,\eta_0)+f_x'(\eta_0,\eta-\eta_0)
\end{equation*}
and
\begin{equation*}
g_r(x,\eta):=\sup_{\eta_0\in \mathbb{Q}^m\cap\overline{B_r(0)}}G_{\eta_0}(x,\eta)\,.
\end{equation*}
We first claim that $g_r(x,\eta)<\infty$ for a.e. $x\in\Omega$ for every $\eta\in\mathbb{R}^m$, $g_r:\,\Om\times\Rm\to\bbR$ is Borel measurable  and that $g_r(x,\cdot):\,\Rm\to \bbR$ is convex for a.e. $x\in\Omega$.

Let $x\in\Omega$ and define $\overline{\xi}$ the element of $\partial f_x(\eta_0)$ such that
\begin{equation}\label{subdiffeta-}
f'_x(\eta_0,\eta-\eta_0)=\langle \overline{\xi},\eta-\eta_0\rangle\,.
\end{equation}
Since $f\in I_{m,p}(\Omega,c_0,c_1,0,a)$, then
\begin{equation}\label{234}
|G_{\eta_0}(x,\eta)|\leq c_1|\eta_0|^p+|\overline{\xi}||\eta-\eta_0|+a(x)
\end{equation}
a.e. $x\in\Omega$, for any $\eta_0,\eta\in\mathbb{R}^m$ and
\begin{equation}\label{estimuniff}
c_0\,|\eta|^p\le\,f(x,\eta)\le\, c_1\,|\eta|^p+a(x)\le\,c_1 2^p\,\,r^p+\|a\|_{L^\infty(\Om)}
\end{equation}
for a.e. $x\in\Om$ and for any $\eta\in\overline{B_{2r}(0)}$.
Moreover, by \eqref{fxlip} and \eqref{estimuniff}, and, arguing as in  \cite[Proposition 4.14]{Cla}, there exists a positive constant $M$, depending only on $c_1$, $\|a\|_{L^\infty(\Om)}$ and $r$, such that 
\[|\overline{\xi}|\leq M\quad\text{for any }\eta_0\in\overline{B_r(0)}
\]
which, together with \eqref{234}, gives
\begin{equation}\label{stimagr}
g_r(x,\eta)\leq c_1r^p+M(|\eta|+r)+\|a\|_{L^\infty(\Om)}<+\infty
\end{equation}
for a.e. $x\in\Omega$ and $\forall \eta\in\mathbb{R}^m$.

Since $g_r$ is a pointwise supremum of the countable family of Borel measurable functions $\Omega\times\Rm\ni (x,\eta)\mapsto G_{\eta_0}(x,\eta)$, with $\eta_0\in \bbQ^m\cap\overline{B_r(0)}$, then it is Borel measurable.
As well, since $g_r(x,\cdot):\,\Rm\to\bbR$ is a pointwise supremum of the countable family of convex functions $\Rm\ni\eta\mapsto G_{\eta_0}(x,\cdot)$ a.e. $x\in\Om$, with $\eta_0\in\bbQ^m\cap\overline{B_r(0)}$, then, by \cite[Proposition 2.20]{Cla}, it is a convex function.

Let us now prove that $g_r$ satisfies \eqref{condG2} and \eqref{condG}. Let $x\in\Omega$ be such that $f(x,\cdot)$ is convex. Then, fixed $\eta_0\in\mathbb{R}^m$,
\begin{align}\label{dgb}
f(x,\eta)\geq f(x,\eta_0)+\langle\xi,\eta-\eta_0\rangle
\end{align}
for any $\eta\in\mathbb{R}^m$ and $\xi\in\partial f_x(\eta_0)$.
Let $\overline{\xi}\in\partial f_x(\eta_0)$ satisfy \eqref{subdiffeta-}. By \eqref{dgb},
\begin{equation*}
f(x,\eta)\geq G_{\eta_0}(x,\eta)
\end{equation*}
and, passing to the supremum, we get
\begin{equation}\label{mezzaG2}
    g_r(x,\eta)\leq\,f(x,\eta)\quad\text{ for a.e. }x\in\Omega\text{ and every }\eta\in\mathbb{R}^m.
\end{equation}
On the other hand, if $\eta\in\overline{B_r(0)}$ and $(\eta_h)_{h\in\mathbb{N}}\subset \mathbb{Q}^m\cap\overline{B_r(0)}$ are such that $\eta_h\to\eta$ as $n\to \infty$, then
\begin{equation}\label{limit}
g_r(x,\eta)\geq f(x,\eta_h)+f_x'(\eta_h,\eta-\eta_h)\ \quad\forall h\in\mathbb{N}.
\end{equation}
Moreover, since
\begin{equation*}
    |f_x'(\eta_n,\eta-\eta_h)|\leq M|\eta-\eta_h|\quad\forall h\in\mathbb{N}\,,
\end{equation*}
we conclude 
\begin{equation}\label{wsx}
\lim_{h\to\infty} f_x'(\eta_h,\eta-\eta_h)=0\,.
\end{equation}
Therefore, by \eqref{mezzaG2}, \eqref{limit}, \eqref{wsx} and by the continuity of $f(x,\cdot)$ in $\mathbb{R}^m$, we obtain \eqref{condG}. 

Fix now $x\in\Omega$ such that both $f(x,\cdot)$ and $g_r(x,\cdot)$ are convex in $\mathbb{R}^m$. By \eqref{condG} and Weierstrass theorem, there exists $\eta_1\in\overline{B_r(0)}$ such that
\begin{align}\label{weie}
    f(x,\eta_1)=\,g_r(x,\eta_1)=\min_{\eta\in\overline{B_r(0)}}g_r(x,\eta)
\end{align}
and, since $f\in I_{m,p}(\Omega,c_0,c_1,0,a)$, then
\begin{align}\label{grpospalla}
    g_r(x,\eta)\geq 0\quad\text{for any }\eta\in\overline{B_r(0)}\,.
\end{align}
Assume, by contradiction, the existence of $\eta_2\in\mathbb{R}^m\setminus\overline{B_r(0)}$ such that 
\begin{equation}\label{eta2}
    g_r(x,\eta_2)<0\,.
\end{equation}
Then, there exist $\eta_3\in\overline{B_r(0)}$ and $\overline{t}\in(0,1)$ such that $\eta_3=\overline{t}\eta_1+(1-\overline{t})\eta_2$ and, since $g_r(x,\cdot)$ is convex in $\mathbb{R}^m$, \eqref{weie} and \eqref{eta2} give
\begin{align*}
    g_r(x,\eta_1)\leq g_r(x,\eta_3)\leq\overline{t}g_r(x,\eta_1)+(1-\overline{t})g_r(x,\eta_2)<g_r(x,\eta_1)\,,
\end{align*}
which yields a contradiction. Then, by \eqref{grpospalla}, we get \eqref{condG2}.

Finally, since $p>1$ and $c_0,M>0$, we have
\begin{equation*}
\lim_{|\eta|\to\infty} \frac{c_0|\eta|^p}{ c_1 r^p+M(|\eta|+r)+\|a\|_{L^\infty(\Om)}}=+\infty
\end{equation*}
and, by \eqref{stimagr}, \eqref{condG3} also follows.
\end{proof}
\begin{lem}\label{lscgeneralized}
Let $f_h\in I_{m,p}(\Om,c_0,c_1,a_0,a_1)$, with $a_0,a_1\in L^\infty(\Om)$, and assume that
\begin{equation*}
f_h(\cdot,\eta)\to f(\cdot,\eta)\quad\text{a.e. in }\Omega,\,\text{for each }\eta\in\Rm.
\end{equation*}
Then,
\begin{itemize}
\item[(i)] $f\in I_{m,p}(\Om,c_0,c_1,a_0,a_1)$;
\item[(ii)] if $(\Phi_h)_h$ weakly converges to $\Phi$ in $L^p(\Om)^m$, then functionals $\mathcal{F}_h,\mathcal{F}:L^p(\Omega)^m\to\mathbb{R}$, defined in \eqref{calFh} and \eqref{calF}, satisfy
\[
\mathcal{F}(\Phi)\leq\liminf_{h\to\infty}\mathcal{F}_h(\Phi_h)\,,\text{ i.e.,}
\]
\[
\int_\Omega f(x,\Phi(x))\,dx\le\,\liminf_{h\to\infty}\int_\Omega f_h(x,\Phi_h(x))\,dx\,.
\]
\end{itemize}
\end{lem}
\begin{proof}(i) It is immediate.

\noindent (ii) Let $(\Phi_h)_h\subset L^p(\Om)^m$ be weakly convergent to $\Phi$ in $L^p(\Om)^m$. Then, there exists a positive constant $M$ such that
\begin{equation*}
    \int_\Omega |\Phi_h|^p\, dx\leq M\quad\text{for any }h\in\mathbb{N}\,.
\end{equation*}
Moreover, since $f\in I_{m,p}(\Om,c_0,c_1,a_0,a_1)$, then $f(\cdot,\Phi(\cdot))\in L^1(\Om)$. Therefore, by the absolute continuity of the Lebesgue's integral, for any $\varepsilon>0$, there exists $\delta=\delta(\varepsilon)>0$ such that
\begin{align}\label{ACl}
\int_A |f(x,\Phi(x))|\, dx<\varepsilon
\end{align}
for any measurable subset $A$ of $\Omega$ such that $|A|<\delta$.

Let us fix $R>0$ and let us consider $\overline{B_R(0)}\subset\mathbb{R}^m$. Then, for any $\varrho>0$, there exist $\eta_1,..,\eta_k\in B_R(0)$ such that
\begin{equation}\label{cupl}
    \overline{B_R(0)}\subset\cup_{i=1}^k B_\varrho(\eta_i)\,.
\end{equation}
Since $f_h,f\in I_{m,p}(\Om,c_0,c_1,a_0,a_1)$, then, by \cite[Theorem 10.4]{Ro}, there exists a positive constant $L_R$ such that
\begin{equation}\label{lipballl}
    \begin{split}
        &|f_h(x,\eta)-f_h(x,\eta_i)|\leq L_R|\eta-\eta_i|\leq L_R\varrho\,,\\
        &|f(x,\eta)-f(x,\eta_i)|\leq L_R|\eta-\eta_i|
        \leq L_R\varrho
    \end{split}
\end{equation}
for any $h\in\mathbb{N}$, $i=1,\dots,k$ and $\eta\in B_\varrho(\eta_i)\cap\overline{B_R(0)}$.

\noindent If $x\in\Om$ and $\eta\in\overline{B_R(0)}$ then, by \eqref{cupl}, there exists $i\in\{1,..,k\}$ such that $\eta\in B_\varrho(\eta_i)$ and, by \eqref{lipballl},
\begin{equation}\label{stimal}
\begin{split}
    |f_h(x,\eta)-f(x,\eta)|\leq 2L_R\varrho+|f_h(x,\eta_i)-f(x,\eta_i)|\,.
\end{split}
\end{equation}
Since $f_h(x,\eta_i)\to f(x,\eta_i)$ for a.e. $x\in\Omega$ and for any $i\in\{1,..,k\}$, then, by Severini-Egoroff theorem, there exist $A_1,..,A_k$, measurable subsets of $\Omega$, such that $|A_i|<\frac{\delta}{2k}$ and such that
\begin{align*}
    \lim_{h\to\infty}\left[\sup_{x\in\Omega\setminus A_i}|f_h(x,\eta_i)-f(x,\eta_i)|\right]=0\,.
\end{align*}
Let $A^\delta:=\cup_{i=1}^k A_i$. Thus, $|A^\delta|<\frac{\delta}{2}$ and
\begin{equation}\label{zh}
    \lim_{h\to\infty}z_h:=\lim_{h\to\infty}\left[\max_{i\in\{1,..,k\}}\sup_{x\in\Omega\setminus A^\delta}|f_h(x,\eta_i)-f(x,\eta_i)|\right]=0\,.
\end{equation}
Therefore, for any $x\in\Omega\setminus A^\delta$ and for any $\eta\in\overline{B_R(0)}$, by \eqref{stimal},
\begin{equation}\label{sevegl}
    |f_h(x,\eta)-f(x,\eta)|\leq 2L_R\varrho+z_h\,.
\end{equation}

Fix $r>0$ and define $\varphi_h:=f_h+a_0$ for any $h\in\mathbb{N}$ and $\varphi:=f+a_0$. Then, trivially, $\varphi_h,\varphi\in I_{m,p}(\Omega,c_0,c_1,0,a_0+a_1)$ and, by Lemma \ref{const}, there exist $R(r)>r$ and $g_r:\Omega\times\mathbb{R}^m\to[0,\infty)$ such that
\begin{align}
\label{condG2'}
&g_{r}(x,\eta)\leq\,\varphi(x,\eta)\quad \mbox{for a.e.  $x\in\Omega$ and every $\eta\in \mathbb{R}^m$}\\
\label{condG'}
&g_{r}(x,\eta)=\,\varphi(x,\eta)\quad \mbox{for a.e. $x\in\Omega$ and every $\eta\in\overline{B_{r}(0)}$}\\
\label{condG3'}
&g_{r}(x,\eta)\leq c_0|\eta|^p\quad \mbox{for a.e. $x\in\Omega$ and every $\eta\in\mathbb{R}^m\setminus\overline{B_R(0)}$}\,.
\end{align}
Notice that, if $x\in\Omega$ and $\eta\in\mathbb{R}^m\setminus \overline{B_R(0)}$, then, by \eqref{condG3'}
\begin{equation}\label{gr1l}
    \varphi_h(x,\eta)\geq c_0|\eta|^p\geq g_r(x,\eta)
\end{equation}
while, if $x\in\Omega\setminus A^\delta$ and $\eta\in\overline{B_R(0)}$, then, by \eqref{sevegl} and \eqref{condG2'}
\begin{equation}\label{gr2l}
    \varphi_h(x,\eta)\geq \varphi(x,\eta)-2L_R\varrho-z_h\geq g_r(x,\eta)-2L_R\varrho-z_h.
\end{equation}
Moreover, since $(\Phi_h)_h$ weakly converges to $\Phi$ in $L^p(\Omega)^m$, then
\begin{equation}\label{intgr}
    \liminf_{h\to\infty}\int_{\Omega\setminus A^\delta}g_r(x,\Phi_h)\,dx\geq\int_{\Omega\setminus A^\delta}g_r(x,\Phi)\,dx\,.
\end{equation}
Therefore, by \eqref{zh}, \eqref{condG'}, \eqref{gr1l}, \eqref{gr2l} and \eqref{intgr}, and by Fatou's lemma
\begin{align*}
    \liminf_{h\to\infty}\int_\Omega\varphi_h(x,\Phi_h)\,dx&\geq\liminf_{h\to\infty}\int_{\Omega\setminus A^\delta}\varphi_h(x,\Phi_h)\,dx\\
    &\geq\liminf_{h\to\infty}\left[\int_{\Omega\setminus A^\delta}g_r(x,\Phi_h)\,dx-(2L_R\varrho+z_h)|\Omega|\right]\\
    &\geq\int_{\Omega\setminus A^\delta}g_r(x,\Phi)\,dx-2L_R\varrho|\Omega|\\
    &\geq\int_{\Omega\setminus( A^\delta\cup\{|\Phi|>r\})}g_r(x,\Phi)\,dx-2L_R\varrho|\Omega|\\
    &=\int_{\Omega\setminus( A^\delta\cup\{|\Phi|>r\})}\varphi(x,\Phi)\,dx-2L_R\varrho|\Omega|\,.
\end{align*}
Moreover, by Chebyshev's inequality,
\begin{align*}
    |\{|\Phi|>r\}|\leq\frac{1}{r^p}\int_\Omega|\Phi(x)|^p\,dx\,.
\end{align*}

Let us choose $r$ such that $|\{|\Phi|>r\}|<\frac{\delta}{2}$. Thus, by \eqref{ACl}
\begin{align*}
    \liminf_{h\to\infty}\int_\Omega\varphi_h(x,\Phi_h)\,dx\geq\int_{\Omega}\varphi(x,\Phi)\,dx-\varepsilon-2L_R\varrho|\Omega|\,,
\end{align*}
that is,
\begin{align*}
    \liminf_{h\to\infty}\int_\Omega f_h(x,\Phi_h)\,dx\geq\int_{\Omega} f(x,\Phi)\,dx-\varepsilon-2L_R\varrho|\Omega|
\end{align*}
and, as $\varepsilon$ and $\varrho$ go to zero, we get the thesis.
\end{proof}
\begin{proof}[Proof of Theorem \ref{GammaWeak}]
By \eqref{3.2}, there exists $\Psi_1:W^{1,p}_{X,0}(\Omega)\to\mathbb{R}$ such that $\Psi_1\leq F_h$ for any $h\in\mathbb{N}$ and 
\begin{align*}
    \lim_{\|u\|_{W^{1,p}_{X,0}(\Omega)}\to\infty}\Psi_1(u)
    =+\infty\,.
\end{align*}
Then, by Theorem \ref{immersion} and \cite[Proposition 8.10]{DM}, we can characterize the $\Gamma$-limit of $(F_h)_h$ in terms of sequences, that is, fixed $u\in W^{1,p}_{X,0}(\Omega)$, it suffices to show that:
\begin{itemize}
    \item [$(a)$] for any $(u_h)_h$ weakly convergent to $u$ in $W^{1,p}_{X,0}(\Omega)$, then
    \[
    F(u)\leq\liminf_{h\to\infty}F_h(u_h)\,;
    \]
    \item [$(b)$] there exists $(v_h)_h$ weakly convergent to $u$ in $W^{1,p}_{X,0}(\Omega)$ such that
    \[
    F(u)=\lim_{h\to\infty}F_h(v_h)\,.
    \]
\end{itemize}
Let $(u_h)_h$ be weakly convergent to $u$ in $W^{1,p}_{X,0}(\Omega)$. Then, $(Xu_h)_h$ weakly converges to $Xu$ in $L^p(\Omega)^m$ and $(a)$ follows, by Lemma \ref{lscgeneralized}.

\noindent Let $v_h:=u$ for any $h\in\mathbb{N}$. Since $(f_h(\cdot,Xu))_h$ converges to $f(\cdot,Xu)$ a.e. in $\Omega$ by hypothesis, then, by the dominated convergence theorem, the sequence $(F_h(u))_h$ converges pointwise to $F(u)$ and $(b)$ also follows.

Similarly, by \eqref{3.2} and \eqref{crescitag}, there exists $\Psi_2:W^{1,p}_X(\Omega)\to\mathbb{R}$ such that $\Psi_2\leq F_h+G$ in $W^{1,p}_X(\Omega)$ for any $h\in\mathbb{N}$ and
\begin{align*}
    \lim_{\|u\|_{W^{1,p}_X(\Omega)}\to\infty}\Psi_2(u)=+\infty\,.
\end{align*}
Then, since $W^{1,p}_X(\Om)$ is compactly embedded in $L^p(\Om)$ by hypothesis, we can characterize the $\Gamma$-limit of $(F_h+G)_h$ in terms of sequences, in virtue of \cite[Proposition 8.10]{DM} and, since $G$ is sequentially continuous in the weak topology of $W^{1,p}_X(\Om)$, then $(F_h+G)_h$ $\Gamma$-converges to $F+G$ in the weak topology of $W^{1,p}_X(\Om)$ by the first part of the proof, and the thesis follows.
\end{proof}
\medskip

An analogous result in the strong topology of $W^{1,p}_X(\Omega)$ still holds true and a proof can be found in \cite[Proposition 2.3.24]{Ma}.
\begin{teo}\label{convpointimplgamma}
Let $\Omega\subset\mathbb{R}^n$ be a bounded open set, $1<p<\infty$ and let $X$ satisfy (LIC) condition. Let $f_h,f\in I_{m,p}(\Om,c_0,c_1,a_0,a_1)$ and let $F_h,\,F:\,W^{1,p}_X(\Om)\to\bbR$ be the corresponding integral functionals, defined as
\begin{align*}
    F_h(u):=\,\int_\Om f_h(x,Xu(x))\,dx\,,\quad
F(u):=\,\int_\Om f(x,Xu(x))\,dx
\end{align*}
for any $u\in W^{1,p}_X(\Om)$ and for any $h\in\mathbb{N}$. Then, $(F_h)_h$ converges pointwise to $F$ in $W^{1,p}_X(\Om)$ if and only if $(F_h)_h$ $\Gamma$-converges to $F$ in the strong topology of $W^{1,p}_X(\Om)$.
\end{teo}


\subsection{\texorpdfstring{$\Gamma$}{TEXT}-convergence in the strong topology of \texorpdfstring{$L^p(\Omega)$}{}}


The first result of this section in an extension of \cite[Theorem 4.11]{MPSC1} to our class of integrands and a proof can be found in \cite[Theorem 2.3.12]{Ma}.
\begin{teo}\label{mainthm} Let $\Om\subset\Rn$ be a bounded open set, let $\mA$ be the class of all open subsets of $\Omega$, $1<p<\infty$ and let $X$ satisfy (LIC) condition. Let $f_h\in I_{m,p}(\Om,c_0,c_1,a_0,a_1)$ and let $ F_h:\,L^p(\Om)\times\mA\to\bbR\cup\{\infty\}$ be the local functional defined as 
\begin{equation}\label{Fstarh}
F_h(u,A):=
\displaystyle{\begin{cases}
\int_{A}f_h(x,Xu(x))dx&\text{ if }A\in\mA,\,u\in W^{1,p}_X(A)\\
\infty&\text{ otherwise}
\end{cases}
\,.
}
\end{equation}
Then, there exist a local functional $F:\,L^p(\Om)\times\mA\to\bbR\cup\{\infty\}$ and $f\in I_{m,p}(\Om,c_0,c_1,a_0,a_1)$ such that, up to subsequences, 
\begin{equation}\label{GammalimitFeuc}
F(\cdot, A)=\,\Gamma(L^p(\Om))\text{-}\lim_{h\to\infty} F_h(\cdot,A)\text{ for each }A\in\mathcal A
\end{equation}
and $F$ admits the following representation 
\begin{equation}\label{GammalimitF}
F(u,A):=
\displaystyle{\begin{cases}
\int_{A}f(x,Xu(x))dx&\text{ if }A\in\mA,u\in W^{1,p}_X(A)\\
\infty&\text{ otherwise}
\end{cases}
\,.
}
\end{equation} 
\end{teo} 
Following \cite[Theorem 21.1]{DM}, using \cite[Theorem 4.16]{MPSC1} instead of \cite[Theorem 19.6]{DM}, we get the $\Gamma$-convergence for functionals with boundary data. A proof can be found in \cite[Theorem 2.3.23]{Ma}.
\begin{teo}\label{convboundary}
Let $\Om\subset\Rn$ be a bounded open set, let $\mA$ be the class of all open subsets of $\Omega$, $1<p<\infty$ and let $X$ satisfy (LIC) condition. Let $f_h\in I_{m,p}(\Om,c_0,c_1,a_0,a_1)$, let $F_h:\,L^p(\Om)\times\mA\to\bbR\cup\{\infty\}$ be the functional in \eqref{Fstarh} and, with a little abuse of notation, denote 
\[
F_h(u):=F_h(u,\Omega)\quad\text{for any }u\in L^p(\Omega)\,.
\]
Fix $\varphi\in W^{1,p}_{X}(\Omega)$ and assume that $(F_h)_h$ $\Gamma$-converges in the strong topology of $L^p(\Omega)$ to $F$ satisfying \eqref{GammalimitF}, with $f\in I_{m,p}(\Om,c_0,c_1,a_0,a_1)$. 
Then, $(F_h+\mathbbm{1}_{\varphi})_h$ $\Gamma$-converges to $F+\mathbbm{1}_{\varphi}$ in the strong topology of $L^p(\Omega)$.
\end{teo}
We conclude this section by providing a $\Gamma$-convergence result of perturbed functionals in the strong topology of $L^p(\Omega)$.
\begin{teo}\label{pertfunct}
Let $\Om\subset\Rn$ be a bounded open set, let $\mA$ be the class of all open subsets of $\Omega$, let $1<p<\infty$ and let $X$ satisfy (LIC) condition. 
Let $f_h,f\in I_{m,p}(\Om,c_0,c_1,a_0,a_1)$ and $F_h,F$ be the functionals in \eqref{Fstarh} and \eqref{GammalimitF} satisfying \eqref{GammalimitFeuc}.
Fixed $\Phi\in L^p(\Om)^m$, let $G^\Phi_h:\,L^p(\Om)\times\mA\to\mathbb{R}\cup\{\infty\}$ be the local functional defined as
\[
G^\Phi_h(u,A):=
\displaystyle{\begin{cases}
\int_{A}f_h(x,Xu(x)+\Phi(x))\,dx&\text{ if }A\in\mA,u\in W^{1,p}_X(A)\\
\infty&\text{ otherwise}
\end{cases}
\,.
}
\]
Then, there exists $G^\Phi:\,L^p(\Om)\times\mA\to\mathbb{R}\cup\{\infty\}$ such that, up to subsequences,
\begin{equation}\label{GammalimitG}
G^\Phi(\cdot,A)=\,\Gamma(L^p(\Om))\text{-}\lim_{h\to \infty}G^\Phi_h(\cdot,A)\text{ for each }A\in\mathcal A
\end{equation}
and $G^\Phi$ admits the following representation
\[
G^\Phi(u,A):=
\displaystyle{\begin{cases}
\int_{A}f(x,Xu(x)+\Phi(x))\,dx&\text{ if }A\in\mA,u\in W^{1,p}_X(A)\\
\infty&\text{ otherwise}
\end{cases}
\,.
}
\]
\end{teo}
Before giving the proof of Theorem \ref{pertfunct}, let us recall the following almost classical result, that we briefly prove for the reader's convenience.
\begin{lem}\label{lem4.5}
Let $f\in I_{m,p}(\Om,c_0,c_1,a_0,a_1)$. Then, there exists a nonnegative constant $c_2$, depending only on $p$ and $c_1$, such that
\begin{equation}\label{ULC}
|f(x,\eta_1)-f(x,\eta_2)|\le\, c_2\,|\eta_1-\eta_2|\,\left(|\eta_1|+|\eta_2|+a(x)^{1/p}\right)^{p-1},
\end{equation}
where $a(x):=\,a_0(x)+a_1(x)$, for a.e. $x\in\Om$, for each $\eta_1,\eta_2\in\Rm$.
\end{lem}
\begin{proof} Without loss of generality, we can assume that $a_0(x)\in\bbR$. Let $\varphi(x,\eta):=\,f(x,\eta)+a_0(x)$ if $(x,\eta)\in\Om\times\Rm$. Then, for a.e. $x\in\Om$ $\varphi(x,\cdot):\,\Rm\to \bbR$ is still convex and there exists a positive constant $\tilde c_1$, depending only on $p$ and $c_1$, such that
\[
c_0\,|\eta|^p\le\,\varphi(x,\eta)\le\,c_1|\eta|^p+a(x)\le \tilde c_1\left(|\eta|+a(x)^{1/p}\right)^p
\]
for a.e. $x\in\Om$, for each $\eta\in\Rm$. Then, arguing as in the proof of \cite[Proposition 2.32]{Da}, there exists a nonnegative constant $c_2$, depending only on $p$ and $c_1$, such that 
\[
\begin{split}
|f(x,\eta_1)-f(x,\eta_2)|&=\,|\varphi(x,\eta_1)-\varphi(x,\eta_2)|\\
&\le\,c_2\,|\eta_1-\eta_2|\,\left(|\eta_1|+|\eta_2|+a(x)^{1/p}\right)^{p-1}
\end{split}
\]
for a.e. $x\in\Omega$, for each $\eta_1,\eta_2\in\Rm$.
\end{proof}
\begin{proof}[Proof of Theorem \ref{pertfunct}] Fix $\Phi\in L^p(\Om)^m$ and, for each $h\in\mathbb{N}$, define $g_h^\Phi(x,\eta):=\,f_h(x,\eta+\Phi(x))$ a.e. $x\in\Om$ and for any $\eta\in\Rm$.
Then,
\begin{equation}\label{ghclass}
    g_h^\Phi\in I_{m,p}(\Om,c_0,c_3,\tilde a_0,\tilde a_1)\,,
\end{equation}
with $\tilde a_0(x):=\,a_0(x)-c_0|\Phi(x)|^p$ and $\tilde a_1(x):=\,a_1(x)+c_3|\Phi(x)|^p$, for a suitable positive constant $c_3$ (depending only on $p$ and $c_1$). 

By Theorem \ref{mainthm}, there exist $G^\Phi:\,L^p(\Om)\times\mA\to\mathbb{R}\cup\{\infty\}$ and
\begin{equation}\label{gPhiinI}
g^\Phi\in I_{m,p}(\Om,c_0,c_3,\tilde a_0,\tilde a_1)
\end{equation}
such that, up to subsequences, \eqref{GammalimitG} holds and $G^\Phi$ can be represented as
\[
G^\Phi(u,A):=
\displaystyle{\begin{cases}
\int_{A}g^\Phi(x,Xu(x))\, dx&\text{ if }A\in\mA,\,u\in W^{1,p}_X(A)\\
\infty&\text{ otherwise}
\end{cases}
\,.
}
\]

To conclude, we show that
\begin{equation}\label{thesis}
G^\Phi(u,A)=\int_{A}f(x,Xu(x)+\Phi(x))\,dx
\end{equation}
for each $A\in\mA$ and $u\in W^{1,p}_{X}(A)$.
We divide the proof of \eqref{thesis} in three steps.
\medskip

\noindent{\bf 1st step.} Let us first prove the existence of a positive constant $c_4$, depending only on $c_0,\,c_1,\,c_2,\,a_0,a_1$ and $p$, such that
\begin{equation}\label{contest}
\begin{split}
|G^{\Phi_1}(u,A)&-G^{\Phi_2}(u,A)|\\&\le\,c_4\,\norma{\Phi_1-\Phi_2}_{L^p}\,\left(\norma{Xu}_{L^p}+\norma{\Phi_1}_{L^p}+\norma{\Phi_2}_{L^p}+1\right)^{p-1}
\end{split}
\end{equation}
for any $\Phi_1,\,\Phi_2\in L^p(\Om)^m$, $A\in\mA$ and $u\in W^{1,p}_X(A)$, where all norms above refer to $A$.
\medskip

Fix $\Phi_1,\,\Phi_2\in L^p(\Om)^m$, $A\in\mA$ and $u\in W^{1,p}_X(A)$. By \eqref{GammalimitG} and \cite[Proposition 8.1]{DM}, there exists a sequence $(u_h)_h\subset L^p(\Omega)\cap W^{1,p}_X(A)$, strongly convergent to $u$ in $L^p(\Omega)$, such that
\begin{equation}\label{uhcoverGPhi2}
G^{\Phi_2}(u,A)=\lim_{h\to\infty}G_h^{\Phi_2}(u_h,A)
\end{equation}
and
\begin{equation}\label{uhcoverGPhi1}
G^{\Phi_1}(u,A)\leq\liminf_{h\to\infty}G_h^{\Phi_1}(u_h,A)\,.
\end{equation}
Then, by \eqref{ULC}, \eqref{ghclass} and H\"older's inequality, there exist positive constants $\alpha_1,\alpha_2$, depending only $c_0,\,c_1,\,c_2,\,a_0,\,a_1$ and $p$, such that
\begin{equation}\label{gphi12}
\begin{split}
G_h^{\Phi_1}(u_h,A)&-G_h^{\Phi_2}(u_h,A)\\&\le\,\int_A|f_h(x,Xu_h+\Phi_1)-f_h(x,Xu_h+\Phi_2)|\,dx\\
&\le\,\alpha_1\,\norma{\Phi_1-\Phi_2}_{L^p}\,\left(\norma{Xu_h}_{L^p}+\norma{\Phi_1}_{L^p}+\norma{\Phi_2}_{L^p}+1\right)^{p-1}\\
&\le\,\alpha_2\,\norma{\Phi_1-\Phi_2}_{L^p}\,\left(G_h^{\Phi_2}(u_h,A)^{1/p}+\norma{\Phi_1}_{L^p}+\norma{\Phi_2}_{L^p}+1\right)^{p-1}
\end{split}
\end{equation}
and, by \eqref{gPhiinI}, \eqref{uhcoverGPhi2} and \eqref{uhcoverGPhi1}, there exists a positive constant $c_4$, depending only on $c_0,\,c_1,\,c_2,\,a_0,a_1$ and $p$, such that
\begin{align*}
G^{\Phi_1}(u,A)&-G^{\Phi_2}(u,A)\\&\le\,\alpha_2\,\norma{\Phi_1-\Phi_2}_{L^p}\,\left(G^{\Phi_2}(u,A)^{1/p}+\norma{\Phi_1}_{L^p}+\norma{\Phi_2}\,_{L^p}+1\right)^{p-1}\\
&\le\,c_4\,\norma{\Phi_1-\Phi_2}_{L^p}\,\left(\norma{Xu}_{L^p}+\norma{\Phi_1}_{L^p}+\norma{\Phi_2}_{L^p}+1\right)^{p-1}.
\end{align*}
By exchanging the roles of $\Phi_1$ and $\Phi_2$, we obtain \eqref{contest}.
\medskip

\noindent{\bf 2nd step.} Let us prove \eqref{thesis}, when $\Phi$ has the form
\begin{equation}\label{simpleformPhi}
\Phi(x)=\,C(x)\,\tilde \Phi(x)\quad\text{a.e. }x\in\Om\,,
\end{equation}
for some $\tilde \Phi\in L^p(\Om)^n$, where $C(x)$ denotes the coefficient matrix of the $X$-gradient \eqref{coeffmatrvf}.
Let us divide this step in three cases.
\medskip

\noindent{\bf Case 1.} Suppose $\tilde\Phi(x)=\xi\in\mathbb{R}^n$ constant and denote $\Phi_\xi(x):=\,C(x)\,\xi$.

\noindent If $u_\xi(x):=\scalare \xi x$ for any $x\in\Rn$, then
\[
G_h^{\Phi_\xi}(u,A)=\,F_h(u+\,u_\xi,A)
\]
and, by \eqref{GammalimitFeuc}, \eqref{GammalimitF} and \eqref{GammalimitG}, we get
\begin{equation}\label{GF}
    G^{\Phi_\xi}(u,A)=\,F(u+\,u_\xi,A)=\,\int_A f(x,Xu(x)+\Phi_\xi(x))\,dx
\end{equation}
for any $A\in\mathcal{A}$ and $u\in W^{1,p}_X(A)$.
\medskip

\noindent{\bf Case 2.} Suppose $\tilde\Phi$ {\it piecewise constant}, i.e., there exist $\xi^1,\dots,\xi^N\in\Rn$ and $A_1,\dots,A_N$ pairwise disjoint open sets such that $|\Om\setminus\cup_{i=1}^N A_i|=\,0$ and
\[
\tilde \Phi(x):=\,\sum_{i=1}^N\chi_{A_i}(x)\,\xi^i.
\]
Fix $A\in\mathcal{A}$ and $u\in W^{1,p}_X(A)$ and denote $\Phi_{\xi^i}(x):=\,C(x)\,\xi^i$. Since $G^{\Phi}(u,\cdot)$ is a measure, then, by additivity on pairwise disjoint open sets and locality, it holds that
\begin{equation}\label{GF1}
    G^{\Phi}(u,A)=\,\sum_{i=1}^N G^{\Phi}(u,A\cap A_i)=\,\sum_{i=1}^N G^{\Phi_{\xi^i}}(u,A\cap A_i)\,.
\end{equation}
Let $\tilde u(x):=\langle\tilde{\Phi}(x),x\rangle$ for any $x\in\Rn$. Then,
\[
\tilde u(x)=u_{\xi^i}(x)=\langle\xi^i,x\rangle\quad\text{a.e. }x\in A\cap A_i
\]
for any $i=1,\dots,N$ and, by locality of $F$,
\begin{equation}\label{GF2}
    F(u+\tilde u,A\cap A_i)=F(u+u_{\xi^i},A\cap A_i)\,.
\end{equation}
Therefore, by \eqref{GammalimitF}, \eqref{GF}, \eqref{GF1}, \eqref{GF2} and additivity of $F$ on pairwise disjoint open sets, we get
\begin{equation}\label{GF4}
\begin{split}
    G^{\Phi}(u,A)&=\,\sum_{i=1}^N G^{\Phi_{\xi^i}}(u,A\cap A_i)=\,\sum_{i=1}^N F(u+u_{\xi^i},A\cap A_i)\\
    &=\,\sum_{i=1}^N F(u+\tilde u,A\cap A_i)=F(u+\tilde u,A)\\
    &=\,\int_{A}f(x,Xu(x)+\Phi(x))\,dx
\end{split}
\end{equation}
for any $A\in\mathcal{A}$ and $u\in W^{1,p}_X(A)$.
\medskip

\noindent{\bf Case 3.} Let $\Phi$ have the form \eqref{simpleformPhi}, let $(\tilde \Phi_j)_j$ be a sequence of piecewise constant functions converging, as $j\to\infty$, to $\tilde\Phi$ strongly in $L^p(\Om)^n$ and define $\Phi_j(x):=C(x)\,\tilde \Phi_j(x)$ a.e. $x\in\Omega$.
Since 
\begin{equation}\label{matrixclimited}
    C\in L^\infty(\Om)^{mn},
\end{equation}
then
\begin{equation}\label{GF0}
    (\Phi_j)_j\text{ strongly converges to }\Phi\text{ in }L^p(\Om)^m.
\end{equation}
If $A\in\mA$ and $u\in W_X^{1,p}(A)$ then, by \eqref{contest},
\begin{equation}\label{GF3}
    G^{\Phi_j}(u,A)\to  G^{\Phi}(u,A)\quad\text{as }j\to\infty
\end{equation}
and, by \eqref{GF4} and H\"older's inequality,  
it holds that
\begin{equation}\label{GF5}
\begin{split}
\big|G^{\Phi_j}(u,A)&-\int_Af(x,Xu+\Phi)\,dx\big|\\
&\le\,\int_A|f(x,Xu+\Phi_j)-f(x,Xu+\Phi)|\,dx\\
&\le\,\alpha_1\,\norma{\Phi_j-\Phi}_{L^p}\,\left(\norma{Xu}_{L^p}+\norma{\Phi_j}_{L^p}+\norma{\Phi}_{L^p}+1\right)^{p-1},
\end{split}
\end{equation}
where $\alpha_1$ is the positive constant given in \eqref{gphi12}.
Therefore, \eqref{thesis} follows from \eqref{GF0}, \eqref{GF3} and \eqref{GF5}.
\medskip
 
\noindent{\bf 3rd step.} Let us finally prove \eqref{thesis} in the general case. 
\medskip

\noindent Fix $\Phi\in L^p(\Om)^m$ and $x\in\Omega_X$ (see Definition \eqref{frkcond}). Then, in virtue of \cite[Lemma 3.3]{MPSC1}, there exists $\tilde\Phi(x)\in\Rn$ such that
\[
C(x)\,\tilde\Phi(x)=\,\Phi(x)
\]
and $\tilde\Phi$ can be represented as
\begin{equation}\label{tildePhirepres}
    \tilde\Phi(x)=\,C(x)^T\,B(x)^{-1}\,\Phi(x)\,,
\end{equation}
where $B(x)$ is the $m\times m$ symmetric invertible matrix defined by
\[
B(x):=\,C(x)\,C(x)^T.
\]
Since $B(x)$ is positive semi-definite for any $x\in\Om$ and it is positive definite if and only if $x\in\Omega_X$, it holds that
\begin{equation}\label{NXnull}
    |\Omega\setminus\Omega_X|:=|\mathcal N_X|=\,0
\end{equation}
and
\begin{equation*}
    \Om_X=\,\left\{x\in\Om:\,{\rm det }B(x)>\,0\right\},\ \ \mathcal N_X=\,\left\{x\in\Om:\,{\rm det }B(x)=\,0\right\}.
\end{equation*}
For any $\eps>0$, define 
\[
\Om_\eps:=\,\left\{x\in\Om: {\rm det }B(x)>\,\eps\right\}.
\]
Since $B\in L^\infty(\Om)^{m^2}$, then, by Cramer's rule and by \eqref{matrixclimited} and \eqref{tildePhirepres},
\begin{equation}\label{B-1Linfty}
    B^{-1}\in L^\infty(\Om_\eps)^{m^2}\quad\text{and}\quad\tilde\Phi\in L^p(\Om_\eps)^n.
\end{equation}
Let $\tilde\Phi_\eps:\Omega\to\Rn$ and $\Phi_\eps:\Omega\to\Rm$ be, respectively, defined by
\[
\tilde\Phi_\eps(x):=
\begin{cases}
\tilde\Phi(x)
&\text{ if }x\in\Omega_\eps\\
0&\text{ if }x\in\Omega\setminus\Omega_\eps
\end{cases}
\]
and
\[
\Phi_\eps(x):=C(x)\tilde\Phi_\eps(x)=
\begin{cases}
\Phi(x)&\text{ if }x\in\Omega_\eps\\
0&\text{ if }x\in\Omega\setminus\Omega_\eps
\end{cases}
\,.
\]
By \eqref{B-1Linfty},
\begin{equation}\label{philp}
    \tilde\Phi_\eps\in L^p(\Omega)^n
\end{equation}
and, by \eqref{NXnull} and H\"older's inequality, it follows that
\begin{equation}\label{PhiepsconvPhi}
\Phi_\eps\to\Phi\quad\text{strongly in }L^p(\Om)^m\text{ as }\eps\to 0\,. 
\end{equation}
If $A\in\mA$ and $u\in W_X^{1,p}(A)$ then, by \eqref{ULC}, \eqref{contest}, \eqref{philp}, by H\"older's inequality and the second step of the proof, there exists a positive constant $c_5$, depending only on $c_0,\,c_1,\,c_2,\,a_0,a_1$ and $p$, such that
\[
\begin{split}
\big|G^{\Phi}(u,A)&-\int_Af(x,Xu+\Phi)\,dx\big|\\
&\le\,\left|G^{\Phi}(u,A)-G^{\Phi_\eps}(u,A)\right|+\,\left|G^{\Phi_\eps}(u,A)-\int_Af(x,Xu+\Phi)\,dx\right|\\
&\le\,c_5\,\|\Phi_\eps-\Phi\|_{L^p(A)^m}\left(\|Xu\|_{L^p(A)^m}+\|\Phi_\eps\|_{L^p(A)^m}+\|\Phi\|_{L^p(A)^m}+1\right)^{p-1}
\end{split}
\]
and \eqref{thesis} follows by \eqref{PhiepsconvPhi}, as $\eps\to 0$.
\end{proof}

\section{Convergence of minima, minimizers and momenta}\label{sec4}


\subsection{Convergence of minima and minimizers}


By Theorem \ref{mainthm}, there exist $F:\,L^p(\Om)\times\mA\to\bbR\cup\{\infty\}$ and $f\in I_{m,p}(\Om,c_0,c_1,a_0,a_1)$ such that, up to subsequences,
\begin{equation*}
F(\cdot,A)=\,\Gamma(L^p(\Om))\text{-}\lim_{h\to\infty} F_h(\cdot,A)\quad\text{for each }A\in\mathcal{A}
\end{equation*}
and $F$ admits the representation \eqref{GammalimitF}. If, in addition, the sequence $( F_h(\cdot, A))_h$ is equicoercive in $L^p(\Omega)$ (see \cite[Definition 7.6]{DM}) then, by \cite[Theorem 7.8]{DM}, $F(\cdot, A)$ attains its minimum in $L^p(\Omega)$ and 
\[
\min_{u\in L^p(\Omega)} F(u,A)=\lim_{h\to\infty} \inf_{u\in L^p(\Omega)} F_h(u,A)\quad\text{for each }A\in\mathcal{A}\,.
\]
Let us now study the asymptotic behaviour of minima and minimizers of problems \eqref{minpb}.
We first need the following preliminary result.
\begin{teo}\label{Corollary 3.1.2}
Let $\Omega\subset\mathbb{R}^n$ be a bounded open set, let $1<p<\infty$ and let $X$ satisfy (LIC) condition and conditions (H1), (H2) and (H3). 
Let $f\in I_{m,p}(\Omega,c_0,c_1,a_0,a_1)$, let $g$ satisfy \eqref{crescitag} and let $F,G:W^{1,p}_X(\Omega)\to\bbR$ be the functionals defined, respectively, by
\[
F(u):=\int_{\Om}f(x,Xu(x))\,dx
\quad\text{and}\quad
G(u):=
\int_{\Om}g(x,u(x))\,dx\,.
\]
Fixed $\varphi\in W^{1,p}_X(\Om)$, let $\Xi,\Xi^{\varphi}:W^{1,p}_X(\Omega)\to\mathbb{R}$ be, respectively, defined by
\[
\Xi:=F+G\quad\text{and}\quad \Xi^{\varphi}:=F+G+\mathbbm{1}_{\varphi}\,.
\]
Then, the minimum problems
\begin{align}\label{proc2}
\min_{u\in\W 1 p (\Om)}\Xi(u)
\end{align}
and 
\begin{align}\label{proc}
\min_{u\in W^{ 1 ,p}_{X,\varphi} (\Om)}\Xi^{\varphi}(u)
\end{align}
have at least a solution, provided that
\begin{equation}\label{constgpos}
0<\,d_0\le\,d_1,
\end{equation}
and
\begin{equation*}
\text{\eqref{constg} holds,}
\end{equation*}
respectively, where $d_0$ and $d_1$ are the constants in \eqref{crescitag}. If, in addition,
\begin{itemize}
\item[(i)] $g(x,\cdot)$ is strictly convex on $\mathbb{R}$ for a.e. $x\in \Omega$, then both solutions in \eqref{proc2} and  \eqref{proc} are unique; 
\item[(ii)] $f(x,\cdot)$ is strictly convex on $\Rm$ and $g(x,\cdot)$ is convex on $\mathbb{R}$ for a.e. $x\in \Omega$, then the solution in \eqref{proc} is unique.
\end{itemize}
Moreover,
\begin{align}\label{proc3}
\min_{u\in\W 1 p (\Om)}\Xi(u)
=\inf_{u\in\Co 1(\Om)\cap\W 1 p (\Om)}\Xi(u)\,.
\end{align}
\end{teo}
\begin{proof}
The proof of \eqref{proc2} and \eqref{proc} relies on the direct method of the calculus of variations. Indeed, existence and uniqueness of  \eqref{proc2} can be proved as in \cite[Theorem 2.6]{DM},  by replacing $W^{1,p}(\Om)$ with $W^{1,p}_X(\Om)$ and  $K$ with  $W^{1,p}_X(\Om)$. Existence and uniqueness of \eqref{proc}  can be proved as in  \cite[Theorem 2.8]{DM}, by replacing $W^{1,p}(\Om)$ with $W^{1,p}_X(\Om)$,  $W^{1,p}_\varphi(\Om)$ with $W^{1,p}_{X,\varphi}(\Om)$, $K$ with  $W^{1,p}_{X,\varphi}(\Om)$ and by using Lemma \ref{PoincareW1pvarphi} instead of \cite[Lemma 2.7]{DM}.

We conclude by proving \eqref{proc3}. Let $(u_h)_h$ be convergent to $u$ in the strong topology of $W^{1,p}_{X}(\Omega)$ and assume that, up to subsequences, $(u_h)_h$ and $(Xu_h)_h$ converge, respectively, to $u$ and $Xu$ a.e. in $\Omega$. Then, $(f(\cdot,Xu_h)+g(\cdot,u_h))_h$ converges to $f(\cdot,Xu)+g(\cdot,u)$ a.e. in $\Omega$ and, by \eqref{3.2} and \eqref{crescitag}, there exist $\varphi_h,\varphi\in L^1(\Omega)$, $h\in\mathbb{N}$, such that
\[
\varphi_h\to\varphi\quad\text{in }L^1(\Omega)\quad\text{as }h\to\infty
\]
and
\begin{align*}
|f(x,Xu_h)+g(x,u_h)|\leq\varphi_h(x)\quad\text{for any }h\in\mathbb{N}\,.
\end{align*}
Therefore, by Pratt's theorem \cite[Theorem A.10]{R}, we can infer that 
\[
\Xi(u_h)\to\Xi(u)\quad\text{as }h\to\infty\,.
\]
Since this holds for a subsequence of any subsequence of the original sequence $(u_h)_h$, then the strong continuity of $\Xi$ follows and, since by Theorem \ref{Theorem 1.2.3} $\ci^{\infty}(\Omega)$ is (strongly) dense in $W^{1,p}_{X}(\Omega)$, we get \eqref{proc3}.
\end{proof}
\begin{oss} Observe that, arguing as in \cite[Corollary 2.9] {DM}, also a linear functional $G:\,L^p(\Om)\to\bbR$, $G(u):=\int_\Om g(x)\,u(x)\,dx$, for a given function $g\in L^{p'}(\Om)$, is allowed in minimization problem \eqref{proc}.
\end{oss}
Before proving Theorem \ref{convminimiz}, we need another technical result.
\begin{lem}\label{coerc}
Let $\Omega$ be a bounded open subset of $\mathbb{R}^n$, let $1<p<\infty$ and let $X$ satisfy (LIC) condition and conditions (H1), (H2) and (H3).
Fixed $\varphi\in W^{1,p}_X(\Omega)$, let $\Psi^\varphi:L^p(\Omega)\to[0,\infty]$ be defined by
\begin{equation}\label{Psiphi}
\Psi^\varphi(u):=\begin{cases}
\int_\Omega |Xu(x)|^p\,dx\quad&\text{if }u\in W^{1,p}_{X,\varphi}(\Omega)\\
\infty\quad&\text{otherwise}
\end{cases}.
\end{equation}
Then, $\Psi^\varphi$ is coercive and lower semicontinuous in the strong topology of $L^p(\Omega)$.
\end{lem}
\begin{proof}
Fix $t\in\mathbb{R}$ and let $(u_h)_h\subset L^p(\Omega)$ be such that
\[
\Psi^\varphi(u_h)\leq t\quad\text{for any }h\in\mathbb{N}\,.
\]
Then, $(u_h)_h$ is bounded in $W^{1,p}_{X,\varphi}(\Omega)$ and, by the reflexivity of $W^{1,p}_X(\Omega)$, there exists $u\in W^{1,p}_X(\Omega)$ such that, up to subsequences,
\begin{align*}
    u_h\to u\quad\text{weakly in } W^{1,p}_X(\Omega)
\end{align*}
and so
\[
u_h-\varphi\to u-\varphi\quad\text{weakly in
}W^{1,p}_X(\Omega)\,.
\]
Since $u_h-\varphi\in W_{X,0}^{1,p}(\Omega)$, which is closed with respect to the weak convergence of $W^{1,p}_X(\Om)$, then $u\in W^{1,p}_{X,\varphi}(\Om)$ and, in virtue of Theorem \ref{immersion}, up to a further subsequence,
\begin{equation*}
u_h\to u\quad\text{strongly in }L^p(\Omega)\,.
\end{equation*}
Thus, the set $\{u\in L^p(\Omega)\,|\,\Psi^\varphi(u)\leq t\}$ is relatively compact in the strong topology of $L^p(\Omega)$, that is, $\Psi^\varphi$ is coercive.

Let now $(u_h)_h$ be such that \begin{equation}\label{conv1}
    u_h\to u\quad\text{strongly in }L^p(\Omega)
\end{equation}
and
\[
\lim_{h\to\infty}\Psi^\varphi(u_h)<\infty\,.
\]
Then, $(u_h)$ is bounded in $W^{1,p}_{X,\varphi}(\Omega)$ and, by the reflexivity of $W^{1,p}_X(\Omega)$, there exists $v\in W^{1,p}_X(\Omega)$ such that, up to subsequences,
\[
u_h\to v\quad\text{weakly in }W^{1,p}_X(\Omega)\,.
\]
By \eqref{conv1},
\[
u=v\quad\text{in }W^{1,p}_X(\Omega)
\]
and, since $\Psi^\varphi$ is lower semicontinuous in the weak topology of $W^{1,p}_X(\Omega)$ (it can be proved as in \cite[Example 2.12]{DM}), then
\[
\Psi^\varphi(u)\leq\lim_{h\to\infty}\Psi^\varphi(u_h)
\]
and the conclusion follows.
\end{proof}
We are now in the position to prove Theorem \ref{convminimiz}.
\begin{proof}[Proof of Theorem \ref{convminimiz}]
{\bf (i)} By \eqref{proc}, both functionals $\Xi^{\varphi}_h$ and $\Xi^{\varphi}$ attain their minima in $L^p(\Omega)$ and, by Theorem \ref{convboundary},
\[
(F_h+\mathbbm{1}_{\varphi})_h\quad\text{$\Gamma$-converges to}\quad F+\mathbbm{1}_{\varphi}
\]
in the strong topology of $L^p(\Omega)$.
Moreover, since $G$ is continuous in the strong topology of $L^p(\Om)$ (it is readily seen, proceeding exactly as in the proof of Theorem \ref{Corollary 3.1.2}), then 
\[
(\Xi^{\varphi}_h)_h\quad\text{$\Gamma$-converges to}\quad\Xi^{\varphi}
\]
in the strong topology of $L^p(\Om)$, in virtue of \cite[Proposition 6.21]{DM}.

Let $\Psi^\varphi$ be the functional defined in \eqref{Psiphi}. By \eqref{3.2}, \eqref{crescitag} and Lemma \ref{PoincareW1pvarphi}, there exist a positive constant $k_3$, depending on $c_0,d_0,k_1$, and a positive constant $k_4$, depending on $a_0,b_0,k_2$, such that
\begin{equation}\label{equic}
    \Xi^\varphi_h(u)\geq    k_3\Psi^\varphi(u)-k_4\quad\text{for any }u\in L^p(\Omega)\text{ for any }h\in\mathbb{N}\,,
\end{equation}
where $k_1,k_2$ are the constants in Lemma \ref{PoincareW1pvarphi}.
Then, by \cite[Proposition 7.7]{DM}, $(\Xi^{\varphi}_h)_h$ is equicoercive in the strong topology of $L^p(\Om)$ and, by \cite[Theorem 7.8]{DM}, $\Xi^{\varphi}$ is also coercive and \eqref{convergenceminima} follows.
\medskip

\noindent {\bf (ii)}
Let $(u_h)_h$ be a sequence of minimizers of $(\Xi^{\varphi}_h)_h$. Without loss of generality, we may assume $(u_h)_h\subset W^{1,p}_{X,\varphi}(\Om)$. As in \eqref{equic},
\begin{align*}
\infty>\Xi^{\varphi}_h(u_h)\geq k_3\|u_h\|_{W^{1,p}_X(\Omega)}^p-k_4\quad\text{for any }h\in\mathbb{N}\,,
\end{align*}
i.e., $(u_h)_h$ in bounded in $W^{1,p}_{X,\varphi}(\Om)$ and, arguing as in the proof of Lemma \ref{coerc}, there exists $\bar u\in W^{1,p}_{X,\varphi}(\Om)$ such that, up to subsequences, \eqref{uhtouLp} holds. Finally, by \cite[Corollary 7.20]{DM}, we get \eqref{barumin}.
\end{proof}


\subsection{Convergence of momenta}


We now deal with the convergence of the momenta associated with functionals $(F_h)_h$. The results contained in this section are inspired by \cite{ADMZ2} and they partially extend those results to integral functionals depending on vector fields.
\medskip

Let $f\in I_{m,p}(\Om,c_0,c_1,a_0,a_1)$, let $f(x,\cdot):\Rm\to\bbR$  
be of class $\ci^1(\Rm)$ for a.e. $x\in\Om$ and denote by $\nabla_\eta f(x,\eta)$ its gradient 
at $\eta\in\mathbb{R}^m$. 

\noindent By Lemma \ref{lem4.5}, there exists a nonnegative constant $c_2$, depending only on $p$ and $c_1$, such that
\[
|\nabla_\eta f(x,\eta)|\le\,c_2\,(2\,|\eta|+(a_0(x)+a_1(x))^{1/p})^{p-1}\quad\text{for a.e. }x\in\Om,\,\forall\eta\in\Rm.
\]
Then, functional $\mathcal F$ in \eqref{calF} is of class $\ci^1(L^p(\Omega)^m)$ and its Gateaux derivative $\partial_\Phi \mathcal F:\,L^p(\Om)^m\to L^{p'}(\Om)^m$ is given by
\begin{equation}\label{GderF}
\partial_\Phi \mathcal F(\Phi)(x)=\nabla_\eta f(x,\Phi(x))\quad\text{for a.e. }x\in\Om\,.
\end{equation}

We are now in position to prove Theorem \ref{convmom}. The proof relies on a technique given in \cite[Lemma 4.11]{DMFT} and \cite[Theorem 4.5]{ADMZ2}.
\begin{proof}[Proof of Theorem \ref{convmom}] 
To conclude, it is sufficient to show that
\begin{equation}\label{thesismom}
\langle \partial_\Phi \mathcal F(Xu),\Psi\rangle_{L^{p'}\times L^p}\le\,\liminf_{h\to\infty}\langle \partial_\Phi \mathcal F_h(Xu_h),\Psi\rangle_{L^{p'}\times L^p}
\end{equation}
for every $\Psi\in L^p(\Om)^m$. 

By (iv), there exist $(u_h)_h$ converging to $u$ in $L^p(\Omega)$ such that
\begin{equation}\label{recovery}
    \mathcal F(Xu)=F(u)=\lim_{h\to\infty}F_h(u_h)=\lim_{h\to\infty}\mathcal F_h(Xu_h)\,.
\end{equation}
If $\Phi\in L^p(\Omega)^m$, then, by Theorem \ref{pertfunct} and \cite[Proposition 8.1]{DM},
\begin{equation}\label{recovery2}
    \mathcal F(Xu+\Phi)=G^\Phi(u)\leq\liminf_{h\to\infty}G_h^\Phi(u_h)=\liminf_{h\to\infty}\mathcal F_h(Xu_h+\Phi)\,.
\end{equation}
Let $\Psi\in L^p(\Omega)^m$ and $(t_j)_j$ be a sequence of positive numbers converging to $0$, as $j\to\infty$. Then, $t_j\,\Psi\in L^p(\Omega)^m$ and, by \eqref{recovery} and \eqref{recovery2}
\[
\frac{\mathcal F\left(Xu+t_j\,\Psi\right)-\mathcal F\left(Xu\right)}{t_j}\le\,\liminf_{h\to\infty}\frac{\mathcal F_h\left(Xu_h+t_j\,\Psi\right)-\mathcal F_h\left(Xu_h\right)}{t_j}
\]
for every $j\in\mathbb{N}$.
Therefore, there exists an increasing sequence of integers $h_j$ such that
\begin{equation}\label{convmom2}
\frac{\mathcal F\left(Xu+t_j\,\Psi\right)-\mathcal F\left(Xu\right)}{t_j}-\frac{1}{j}\le\,\frac{\mathcal F_h\left(Xu_h+t_j\,\Psi\right)-\mathcal F_h\left(Xu_h\right)}{t_j}
\end{equation}
for every $h\ge\,h_j$.
If $\eps_h:=\,t_j$ for $h_j\le\,h<\,h_{j+1}$ then, by \eqref{convmom2},
\begin{equation}\label{mom1}
    \liminf_{h\to\infty}\frac{\mathcal F\left(Xu+\eps_h\,\Psi\right)-\mathcal F\left(Xu\right)}{\eps_h}\le\,\liminf_{h\to\infty}\frac{\mathcal F_h\left(Xu_h+\eps_h\,\Psi\right)-\mathcal F_h\left(Xu_h\right)}{\eps_h}\,.
\end{equation}

By (ii) and (iii), both $\mathcal F_h$ and $\mathcal F$ are of class $\ci^1(L^p(\Omega)^m)$. Therefore,
\begin{equation}\label{mom2}
    \langle \partial_\Phi \mathcal F(Xu),\Psi\rangle_{L^{p'}\times L^p}=\,\lim_{h\to\infty}\frac{\mathcal F\left(Xu+\eps_h\,\Psi\right)-\mathcal F\left(Xu\right)}{\eps_h}
\end{equation}
and, by mean value theorem, there exists $\tau_h\in (0,\eps_h)$ such that
\begin{equation}\label{mom3}
    \frac{\mathcal F_h\left(Xu_h+\eps_h\,\Psi\right)-\mathcal F_h\left(Xu_h\right)}{\eps_h}=\,\langle \partial_\Phi \mathcal F_h(Xu_h+\tau_h\Psi),\Psi\rangle_{L^{p'}\times L^p}\,.
\end{equation}
Finally, by (i) and \cite[Lemma 4.4]{ADMZ2}, with $H_h=\,\nabla_\eta f_h$, $\Phi_h=\,Xu_h$, $\Psi_h=\,\tau_h\,\Psi$ and $\Phi\equiv 1$, it holds that
\[
\liminf_{h\to\infty}\langle \partial_\Phi \mathcal F_h(Xu_h+\tau_h\Psi),\Psi\rangle_{L^{p'}\times L^p}=\,\liminf_{h\to\infty}\langle \partial_\Phi \mathcal F_h(Xu_h),\Psi\rangle_{L^{p'}\times L^p}
\]
and \eqref{thesismom} follows from \eqref{mom1}, \eqref{mom2} and \eqref{mom3}.
\end{proof}
\begin{oss} We are aware that assumption (iii) in Theorem \ref{convmom}, on integrand function $f(x,\cdot)$ to be $\ci^1(\Rm)$, is quite strong. Since the techniques exploited in the Euclidean setting do not seem to work in our framework, we do not know whether  assumption (ii) actually  implies it (see \cite[Proposition 3.5]{GP} and \cite[Theorem 2.8]{ADMZ2}). However, we will be able to prove this result in two relevant cases: the periodic homogenization in Carnot groups and the case of quadratic forms (see sections \ref{sec5} and \ref{sec6} below).
\end{oss}
\begin{proof}[Proof of Corollary \ref{corconvmom}]
Let $(u_h)_h$ be a sequence of minimizers of $(\Xi_h^\varphi)_h$. By Theorem \ref{convminimiz}, there exists a minimizer $u$ of $\Xi^{\varphi}$ such that, up to subsequences,
\begin{equation*}
u_h\to u\text{ weakly in } W^{1,p}_{X}(\Omega)\text{  and strongly in }L^p(\Om)
\end{equation*}
and
\begin{equation}\label{corconvmom2}
\Xi_h^\varphi(u_h)\to \Xi^\varphi(u)\,.
\end{equation}
Since $G$ is continuous in the strong topology of $L^p(\Om)$, then
\begin{equation*}
G(u_h)\to G(u)
\end{equation*}
and, by \eqref{corconvmom2},
\begin{equation}\label{corconvmom4}
\mathcal F_h(Xu_h)=\,F_h(u_h)\to F(u)=\mathcal F(Xu)\,.
\end{equation}
Then, the thesis follows by Theorem \ref{convmom}.
\end{proof}


\section{Convergence of minimizers and momenta in the case of homogenization in Carnot groups}\label{sec5}


We are going to study  the convergence of momenta  for  a $\Gamma$-convergent sequence of functionals  in the case of the periodic {\it homogenization} in Carnot groups. The asymptotic behaviour for the periodic homogenization of sequences of  functionals and differential operators  in Carnot groups has been the object of an intensive study in the last two decades. Here, for the sake of simplicity, we will restrict to the case of periodic homogenization in the setting of Heisenberg groups, which turn out to be the simplest Carnot groups. 

Let us recall that the $s$-dimensional Heisenberg group $\bbH^s:=\,(\bbR^{n},\cdot)$, with $n=2s+1=m+1$ and $s\ge\,1$  integer, is a Lie group with respect to the group law
\[
x\cdot y:=\,\left(x_1+y_1,\dots,x_{m}+y_{m}, x_{m+1}+y_{m+1}+\omega(x,y)\right)
\]
where
\[
\omega(x,y):=\,\frac{1}{2}\sum_{i=1}^s(x_i\,y_{s+i}-y_i\,x_{s+i})
\]
if $x=(x_1,\dots,x_{m+1})$ and  $y=(y_1,\dots,y_{m+1})$ are in $\Rn$. Moreover,  $\bbH^s$ can be equipped with a one-parameter family of {\it intrinsic dilations} $\delta_\lambda:\,\Rn\to\Rn$ ($\lambda>\,0$),  defined as 
\[
\delta_\lambda(x):=\,(\lambda\,x_1, \dots, \lambda\,x_m, \lambda^2\,x_{m+1}) \text{ if }x=(x_1,\dots,x_m,x_{m+1})\,,
\]
which are also automorphisms of the group.
A standard basis of the Lie algebra associated to $\bbH^s$ is given by the following family of $n$ left-invariant vector fields
\[
X_j:=\begin{cases}
\partial_j-\frac{x_{s+j}}{2}\partial_n&\text{ if }1\le\,j\le\,s\\
\partial_j+\frac{x_{j-s}}{2}\partial_n&\text{ if }s+1\le\,j\le\,m=2s\\
\partial_n&\text{ if }j=\,n=m+1
\end{cases}
\,.
\]
Notice that the only nontrivial commutations among vector fields $X_j$'s are given by
\begin{equation}\label{commHeis}
[X_j,X_{s+j}]=\,X_n\text{ for each }j=1,\dots,s\,.
\end{equation}
The family of vector fields
\begin{equation}\label{horgradheis}
X:=\,(X_1,\dots,X_m)
\end{equation}
is called {\it horizontal gradient} of the Heisenberg group $\bbH^s$. It is immediate to see that $X$ satisfies (LIC) condition and, by \eqref{commHeis}, H\"ormander condition. Thus, by Remark \ref{exvf} (i), the horizontal gradient $X$ satisfies assumptions (H1), (H2) and (H3).

Since $\bbH^s$ is a homogeneous Lie group, it is possible to introduce two natural notions of intrinsic periodicity and linearity, with respect to its algebraic structure (see \cite[Definition 2.3]{DDMM}). 
\begin{defi}
A function $g:\,\bbR^n\to\bbR$ is said to be {\it $H$-periodic} whenever
\begin{equation}\label{Hper}
g(2k\cdot x)=\,g(x)\quad\text{for any }x\in\bbR^n\text{ for any }k\in\bbZ^n.
\end{equation}
Moreover, a function $l:\,\Rn\to\bbR$ is said to be {\it $H$-linear} if 
\begin{equation*}
    l(x\cdot y)=\,l(x)+\,l(y)\quad\text{and}\quad l(\delta_{\lambda} (x))=\,\lambda\,l(x)
\end{equation*}
for each $x,\,y\in\Rn$ and $\lambda>\,0$. 
\end{defi}
It is well-known that each $H$-linear function $l:\,\Rn\to\bbR$ can be represented by a unique $\eta\in\Rm$ in such a way
\begin{equation}\label{hlinear}
    l(x)=\,l_\eta(x):=\,\langle \eta,\pi_m(x)\rangle
\end{equation}
where $\langle\cdot,\cdot\rangle$ and $\pi_m:\Rn\to\Rm$ denote, respectively, the scalar product on $\Rm$ and the projection map
\[
\pi_m(x):=\,(x_1,\dots,x_m)\quad\text{if } x=(x_1,\dots,x_{m+1})\in\Rn.
\]
\begin{defi}\label{affine}
A function $u:\,\Rn\to\bbR$ is said to be {\it $H$-affine} if there exist $\eta\in\Rm$ and $a\in\bbR$ such that
\begin{equation*}
u(x)=\,l_\eta(x)+a\quad\text{for each }x\in\Rn.
\end{equation*}
\end{defi}
Let $X=(X_1,\dots,X_m)$ denote the horizontal gradient \eqref{horgradheis} and fix $\eta\in\Rm$. 
One can prove that, for any $H$-affine function $u$
\begin{equation}\label{affinefunction}
Xu(x) =\,\eta\quad\text{if and only if}\quad u(x)=\,l_\eta(x)+a
\end{equation}
for some $a\in\bbR$ and for any $x\in\mathbb{R}^n$ (see e.g. \cite[Lemma 3.1]{DDMM}).

Let us recall a $\Gamma$-convergence result for the periodic homogenization of a sequence of integral functionals in $\bbH^s$.
\begin{teo}[{\cite[Theorem 5.2]{DDMM}}]\label{resultDDMM}
Let $\mA_0$ be the class of all open bounded subsets of $\mathbb{R}^n$, $1<p<\infty$ and $f\in I_{m,p}(\bbR^{n},c_0,c_1,0,1)$ satisfy
\begin{equation}\label{Hperx}
f(\cdot,\eta):\,\Rn\to[0,\infty)\text{ is $H$-periodic for every }\eta\in\Rm.
\end{equation} 
Let $F_\eps:\,L^p_{\rm{loc}}(\mathbb{R}^n)\times\mA_0\to[0,\infty]$ be the local functional defined as 
\begin{equation}\label{defFeps}
F_\eps(u,A):=
\displaystyle{\begin{cases}
\int_{A}f(\delta_{1/\eps}(x),Xu(x))\, dx&\text{ if }A\in\mA_0,\,u\in W^{1,p}_X(A)\\
\infty&\text{ otherwise}
\end{cases}
\,.
}
\end{equation}
Then, there exist a local functional $F_0:\,L^p_{\rm{loc}}(\mathbb{R}^n)\times\mA_0\to[0,\infty]$ and a convex function $f_0:\mathbb{R}^m\to[0,\infty)$, not depending on $x$ and satisfying
\[
c_0|\eta|^p\leq f_0(\eta)\leq c_1|\eta|^p+1\quad\text{for any }\eta\in\mathbb{R}^m,
\]
such that, up to subsequences,
\begin{equation}\label{FepshGammaconvF0}
F_0(\cdot,A)=\,\Gamma(L^p_{\rm{loc}}(\Rn))\text{-}\lim_{h\to \infty}F_{\eps_h}(\cdot,A)\text{ for each }A\in\mA_0
\end{equation}
for any infinitesimal sequence $(\eps_h)_h$, and $F_0$ can be represented as
\begin{equation}\label{defF0}
F_0(u,A):=
\displaystyle{\begin{cases}
\int_{A}f_0(Xu(x))\,dx&\text{ if }A\in\mA_0,\,u\in W^{1,p}_X(A)\\
\infty&\text{ otherwise}
\end{cases}
\,.
}
\end{equation}
\end{teo}
We are going to prove the following regularity result for the integrand function $f_0$ which represents the $\Gamma$-limit. It is an extension to functionals depending on vector fields of \cite[Proposition 3.5]{GP}, when $X=\,D$, and of \cite[Theorem 2.8]{ADMZ2}, which applies to a more general setting.
\begin{prop}\label{GammalimitC1} Under the hypotheses of Theorem \ref{resultDDMM}, suppose that
\begin{equation}\label{fC1}
    \Rm\ni\eta\mapsto f(x,\eta)\text{ belongs to }\ci^1(\Rm)\text{ for a.e. }x\in\Om
\end{equation}
and, fixed $0\leq\alpha\leq\min\{1,p-1\}$, that there exists a positive constant $\overline{c}$ such that
\begin{equation}\label{estimgrad}
    |\nabla_\eta f(x,\eta_1)-\nabla_\eta f(x,\eta_2)|\le\,\overline{c}\,|\eta_1-\eta_2|^{\alpha}\,\left(|\eta_1|+|\eta_2|+1\right)^{p-1-\alpha}
\end{equation}
a.e. $x\in\Omega$, for any $\eta_1,\,\eta_2\in\Rm$. Then, $f_0\in\ci^1(\Rm)$.
\end{prop}
\begin{proof} We follow the same strategy of \cite[Proposition 3.5]{GP}.
\medskip

Let $f_h:\,\Rn\times\Rm\to [0,\infty)$ be the function
\begin{equation}\label{fhper}
    f_h(y,\eta):=\,f(\delta_{1/\eps_h}(y),\eta)\quad\text{for any }y\in\Rn,\eta\in\Rm,h\in\mathbb{N}
\end{equation}
and denote $B_1$ the unit ball in $\mathbb{R}^n$ centered at $0$. Fixed $\eta\in\mathbb{R}^m$, let $u:\mathbb{R}^n\to\mathbb{R}$ be the $H$-affine function \eqref{hlinear}, that is,
\[
u(y)=l_\eta(y)\quad\text{for any }y\in\mathbb{R}^n.
\]
If $(u_h)_h$ is a recovery sequence for $u$, then, by \eqref{FepshGammaconvF0} and \eqref{defF0},
\begin{equation}\label{uhtoleta}
    u_h\to l_\eta\quad\text{strongly in }L^p(B_1)
\end{equation}
and
\begin{equation}\label{FepshtoF0eta}
    F_0(l_\eta,B_1)=\,f_0(\eta)\,|B_1|=\lim_{h\to\infty}F_{\eps_h}(u_h,B_1)\,.
\end{equation}
By \eqref{FepshtoF0eta} and $(I_3)$, $(Xu_h)_h$ is bounded in $L^p(B_1)^m$ and, by \eqref{ULC},
\[
|\nabla_\eta f_h(y,Xu_h(y))|\le\, c_2
\left(2|Xu_h(y)|+1\right)^{p-1}
\]
for each $h\in\mathbb{N}$ and for a.e. $y\in B_1$. Then, $(\nabla_\eta f_{h}(\cdot,Xu_h))_h$ is bounded in $L^{p'}(B_1)^m$ and there exists $\psi\in\Rm$ such that, up to subsequences,
\begin{equation}\label{GammalimitC11}
    \frac{1}{|B_1|}\int_{B_1}\nabla_\eta f_{h}(y,Xu_h(y))\,dy\to\psi\,.
\end{equation}
Let $(t_j)_j$ be a decreasing infinitesimal sequence. By the convexity of $f_{h}$ in the second variable,
\begin{equation}\label{convcarn}
\begin{split}
\int_{B_1}f_h(y,Xu_h(y)+t_j\zeta)&-f_{h}(y,Xu_h(y))\,dy\\
&\le\, t_j\,\int_{B_1}\langle\nabla_\eta f_{h}\left(y,Xu_h(y)+t_j \zeta\right),\zeta\rangle\,dy
\end{split}
\end{equation}
for any $j\in\bbN$ and $\zeta\in\bbR^m$.

By the $\Gamma\text{-}\liminf$ inequality and \eqref{FepshtoF0eta}, we can find $h_j\in\mathbb{N}$ such that
\[
\begin{split}
&\frac{f_0\left(\eta+t_j\zeta\right)-f_0\left(\eta\right) }{t_j}-\frac{1}{j}\le\,\frac{1}{|B_1|}\int_{B_1}\langle\nabla_\eta f_{{h_j}}\left(y,Xu_{h_j}(y)+t_j \zeta)\right),\zeta\rangle\,dy
\end{split}
\]
and so
\begin{equation}\label{GammalimitC12}
\begin{split}
\limsup_{j\to\infty}\,&\frac{f_0\left(\eta+t_j\zeta\right)-f_0\left(\eta\right)}{t_j}\\
&\le\,\frac{1}{|B_1|}\limsup_{j\to\infty}\int_{B_1}\langle\nabla_\eta f_{{h_j}}\left(y,Xu_{h_j}(y)+t_j \zeta\right),\zeta\rangle\,dy\,.
\end{split}
\end{equation}
By \eqref{estimgrad}, \eqref{GammalimitC11} and \cite[Lemma 4.4]{ADMZ2}, with $H_j:=\nabla_\eta f_{h_j}$, $\Phi_j:=Xu_{h_j}$, $\Psi_j:=\,t_j\,\zeta$ and $\Phi\equiv 1$, we can infer that
\[
\begin{split}
\lim_{j\to\infty}\int_{B_1}&\langle\nabla_\eta f_{{h_j}}\left(y,Xu_{h_j}(y)+t_j \zeta\right),\zeta\rangle\,dy\\&=\,\lim_{j\to\infty}\int_{B_1}\langle\nabla_\eta f_{{h_j}}\left(y,Xu_{h_j}(y)\right),\zeta\rangle\,dy=\,\langle \psi,\zeta\rangle\,|B_1|\,.
\end{split}
\]
By \eqref{GammalimitC12}, for every subgradient $v$ of the convex function $f_0$ at $\eta$
\[
\langle v,\zeta\rangle\le\,\limsup_{j\to\infty}\frac{f_0\left(\eta+t_j\zeta\right)-f_0\left(\eta\right) }{t_j}\le\,\langle \psi,\zeta\rangle\quad\text{for each }\zeta\in\Rm\,.
\]
Then, $v=\psi$ and there is a unique subgradient of $f_0$ at $\eta$. Therefore, $f_0$ is differentiable at $\eta$ for each $\eta\in\Rm$, by \cite[Theorem 25.1]{Ro}. On the other hand, by \cite[Corollary 25.5.1]{Ro}, any finite convex differentiable function on an open convex set must be of class $\ci^1$.
\end{proof}
\begin{cor}\label{corconvmomhomog} 
Let $\Omega\subset\Rn$ be a bounded and connected open set, let $1<p<\infty$ and let $X$ be the horizontal gradient defined in \eqref{horgradheis}.
Let $f\in I_{m,p}(\Omega,c_0,c_1,0,1)$ satisfy \eqref{Hperx}, \eqref{fC1} and \eqref{estimgrad}, let $f_h$ be as in \eqref{fhper}, let $g$ satisfy \eqref{crescitag} and \eqref{constg}, let $G$ be the functional in \eqref{Gdatum} and, with a little abuse of notation, denote 
\[
F_{\eps_h}(u):=F_{\eps_h}(u,\Omega)\quad\text{and}\quad F_0(u):=F_0(u,\Omega)\quad\text{for any }u\in L^p(\Omega)\,,
\]
where $F_{\eps_h}$ and $F_0$ are, respectively, as in \eqref{defFeps} and \eqref{defF0} such that \eqref{FepshGammaconvF0} holds for each infinitesimal sequence $(\eps_h)_h$.
Fixed $\varphi\in W^{1,p}_{X}(\Omega)$, consider functionals $\Xi_h^{\varphi},\Xi_0^{\varphi}:L^p(\Omega)\to\mathbb{R}\cup\{\infty\}$, defined by
\[
\Xi_h^{\varphi}:=F_{\eps_h}+G+\mathbbm{1}_{\varphi}\quad\text{and}\quad \Xi_0^{\varphi}:=F_0+G+\mathbbm{1}_{\varphi}\,.
\]
If $(u_h)_h\subset L^p(\Omega)$ is a sequence of minimizers of $(\Xi_h^{\varphi})_h$, then there exists a minimum $u$ of $\Xi^{\varphi}$ such that, up to subsequences,
\begin{equation*}
    u_h\to u\text{ weakly in }W^{1,p}_X(\Omega)\text{ and strongly in }L^p(\Om)\,.
\end{equation*}
Moreover, the convergence of momenta also holds, that is,
\begin{equation}\label{convmom0hom}
\nabla_\eta f_h(\cdot,Xu_h(\cdot))\to \nabla_\eta f_0(Xu(\cdot))
\end{equation}
weakly in $L^{p'}(\Om)^m$ as $h\to\infty$.
\end{cor}
\begin{proof} In virtue of Remark \ref{exvf} (i), by choosing $\Om_0=\,\Rn$, the horizontal gradient $X$ satisfies (H1), (H2) and (H3) as well as (LIC) condition.
By \eqref{FepshGammaconvF0}, $(F_{\eps_h})_h$ $\Gamma$-converges to $F_0$ in the strong topology of $L^p(\Omega)$.
Therefore, by Proposition \ref{GammalimitC1}, we can apply Corollary \ref{corconvmom} and we get the desired conclusions.
\end{proof}


\section{\texorpdfstring{$H$}{TEXT}-convergence for linear operators in \texorpdfstring{$X$}{TEXT}--divergence form}\label{sec6}


Throughout this section $X=\,(X_1,\dots,X_m)$ denotes a family of Lipschitz continuous vector fields on an open neighborhood $\Omega_0$ of $\Omega$, open subset of $\mathbb{R}^n$. Moreover, we denote by $H^{1}_{X,0}(\Omega)$ the space $W^{1,2}_{X,0}(\Omega)$ and by $H^{-1}_{X}(\Omega)$ its dual space.
Since $H^{1}_{X,0}(\Omega)$ turns out to be a Hilbert space, in a standard way, then we can construct the triplet
\[
H^{1}_{X,0}(\Omega)\subset L^2(\Om)\subset H^{-1}_{X}(\Omega)\,,
\]
with the space $L^2(\Om)$ as {\it pivot space}.

Let $a(x):=\,[a_{ij}(x)]$ be a $m\times m$ symmetric matrix, with $a_{ij}\in L^{\infty}(\Omega)$ for every $i,j\in \{1,\ldots, m\}$ and assume the existence of $c_0\leq c_1$, positive constants, such that
\begin{equation}\label{condJ1}
c_0|\eta|^2\leq \langle a(x)\eta,\eta\rangle \leq c_1|\eta|^2\quad\text{a.e. }x\in\Omega\text{ and for any }\eta\in\mathbb{R}^m.
\end{equation}
As a consequence of Corollary \ref{corconvmom}, we can infer a {\it $H$-compactness} result for the class of linear partial differential operators in $X$-divergence form  
\begin{equation}\label{operL}
\mathcal L=\,{\rm div}_X(a(x)X):=\sum_{j,i=1}^mX_j^T(a_{ij}(x) X_i)\,,
\end{equation}
whose domain $D(\mathcal{L})$ is the set of functions $u\in W^{1,2}_X(\Omega)$ such that the distribution defined by the right hand side belongs to $L^2(\Omega)$. Here, $X^T_j:=-(\mathrm{div}(X_j)+X_j)$ denotes the (formal) adjoint of $X_j$ in $L^2(\Omega)$, as in \eqref{XjT}. In according with \cite[Chapter 13]{DM}, we denote this class of operators as $\mathcal{E}(\Omega):=\mathcal{E}(\Omega,c_0,c_1)$.
\begin{oss}\label{equivsoluzmin}
For each $\mathcal L\in\mathcal{E}(\Om)$ and for each $\mu\ge\,0$ the linear operator 
\[
\mu\,{\rm Id}+\mathcal L:\,H^{1}_{X,0}(\Omega)\to H^{-1}_X(\Omega)
\]
is coercive and then it is an isomorphism. Moreover, it is well-known that, fixed $g\in L^2(\Om)$, $u=(\mu\,{\rm Id}+\mathcal L)^{-1}g$ is the (unique) solution to
\[\displaystyle{\begin{cases}
\mu v+\mathcal{L}(v)=g\ \text{in}\ \Omega\\
v\in H^1_{X,0}(\Omega)
\end{cases}
}
\]
and the (unique) minimizer of $F+G:\,H^{1}_{X,0}(\Omega)\to\bbR$, where
\[
F(v)=\,\frac{1}{2}\int_{\Omega}\langle a(x)Xv(x),Xv(x)\rangle\,dx\,,\quad G(v):=\int_{\Omega}\left(\frac{\mu}{2} v^2-gv\right)dx\,,
\]
if $v\in H^{1}_{X,0}(\Omega)$.
\end{oss}
\bigskip

\noindent We are going to prove the following $H$-compactness result for operators belonging to $\mathcal{E}(\Omega)$.
\begin{teo} \label{FDC} Let $\Omega\subset\mathbb{R}^n$ be a bounded open set, let $1<p<\infty$ and let $X$ satisfy (LIC) condition and conditions (H1), (H2), (H3).
Let $\mathcal L_h\in\mathcal{E}(\Omega)$ and let $a^h(x)=[a^h_{ij}(x)]$ be the associated matrix, in according with \eqref{operL}.
Then, there exist a symmetric matrix $a\,=[a_{ij}(x)]$, satisfying \eqref{condJ1}, and an operator $\mathcal{L}_\infty:=\,{\rm div}_X(a(x)X)\in\mathcal{E}(\Omega)$ such that, for any $g\in L^2(\Omega)$, $\mu\geq 0$ and $h\in\mathbb{N}$, if $u_h$ and $u_\infty$ denote, respectively, the (unique) solutions to
\[
\displaystyle{\begin{cases}
\mu u+\mathcal{L}_h(u)=g\ \text{in}\ \Omega\\
u\in H^1_{X,0}(\Omega)
\end{cases}
}
\quad\text{and}\qquad
\displaystyle{\begin{cases}
\mu u+\mathcal{L}_\infty(u)=g\ \text{in}\ \Omega\\
u\in H^1_{X,0}(\Omega)
\end{cases}
}
\]
then, up to subsequences, the following convergences hold:
\begin{equation}\label{convsol}
u_h\to u_\infty\text{ strongly in }L^2(\Om)\quad\text{(convergence of solutions)}
\end{equation}
and 
\begin{equation*}
a^hXu_h\to aXu_\infty\text{ weakly in }L^2(\Om)^m\quad\text{(convergence of momenta)}.
\end{equation*}
\end{teo}
\begin{proof}
{\bf 1st step.} Let us prove that, up to a subsequence, there exists an operator $\mathcal L=\,{\rm div}_X(a(x)X)\in  \mathcal{E}(\Omega)$ for which \eqref{convsol} holds. 
\medskip

Let $(a^h)_h$ be the sequence of matrices associated to $(\mathcal{L}_h)_h$ and let $F_h: L^2(\Omega)\to [0,\infty]$ be the quadratic functional defined by
\[
F_h(u):=\displaystyle{\begin{cases}
\frac{1}{2}\int_{\Omega}\langle a^h(x)Xu(x),Xu(x)\rangle\,dx&\text{ if }\,u\in W^{1,2}_{X}(\Omega)\\
\infty&\text{ otherwise}
\end{cases}
\,.
}
\]
By \cite[Theorem 4.20]{MPSC1}, there exist $F:L^2(\Omega)\to[0,\infty]$ and a symmetric matrix $a\,=[a_{ij}(x)]$, satisfying \eqref{condJ1}, such that, up to subsequences, $(F_h)_h$ $\Gamma$-converges in the strong topology of $L^2(\Omega)$ to $F$ and $F$ can be represented as
\[
F(u):=\displaystyle{\begin{cases}
\frac{1}{2}\int_{\Omega}\langle a(x)Xu(x),Xu(x)\rangle\,dx&\text{ if }\,u\in W^{1,2}_{X}(\Omega)\\
\infty&\text{ otherwise}
\end{cases}
\,.
}
\]
Let $\mathcal{L}_\infty$ be the elliptic operator associated with $a$ on $L^2(\Omega)$, as in \eqref{operL}. It is easy to see that $\mathcal{L}_\infty$ is the operator associated, in the sense of \cite[Definition 12.8]{DM}, to functional $F^0: L^2(\Omega)\to [0,\infty]$ defined by
\[
F^0(u)=\displaystyle{\begin{cases}
\frac{1}{2}\int_{\Omega}\langle a(x)Xu(x),Xu(x)\rangle\,dx & \text{ if }\,u\in H^{1}_{X,0}(\Omega)\\
\infty&\text{ otherwise}
\end{cases}
\,.
}
\]
Let us consider $F^0_h: L^2(\Omega)\to [0,\infty]$ defined by
\[
F^0_h(u)=\displaystyle{\begin{cases}
\frac{1}{2}\int_{\Omega}\langle a^h(x)Xu(x),Xu(x)\rangle\,dx & \text{ if }\,u\in H^{1}_{X,0}(\Omega)\\
\infty&\text{ otherwise}
\end{cases}
\,,
}
\]
whose associated operators are the $(\mathcal{L}_h)_h$.
By Theorem \ref{convboundary}, with $\varphi=0$ and $A=\Omega$, it holds that
\begin{equation}\label{Fh0GammaF0}
(F_h^0)_h \text { $\Gamma$-converges to }F^0\text{ in the strong topology of }L^2(\Omega)\,.
\end{equation}
Fix $\mu\geq 0$ and $g\in L^2(\Omega)$, and denote by $G: L^2(\Omega)\to \mathbb{R}$ the functional
\[
G(u):=\int_{\Omega}\left(\frac{\mu}{2} u^2-gu\right)dx\,.
\]
Since $G$ is (strongly) continuous in $L^2(\Omega)$, then, by \eqref{Fh0GammaF0} and in virtue of \cite[Proposition 6.21]{DM},
\begin{equation}\label{Fh0+PhiGammaF0+Phi}
(F_h^0+G)_h \text { $\Gamma$-converges to }F^0+G\text{ in the strong topology of }L^2(\Omega)\,.
\end{equation}
By Remark \ref{equivsoluzmin}, $u_h$ and $u$ turn out to the unique minimizers of $F_h+G$ and $F+G$, respectively.
Therefore, by Theorem \ref{convminimiz}, we get \eqref{convsol} and
\begin{align*}
\lim_{h\to\infty}\left(F_h^0(u_h)+G(u_h)\right)&=\lim_{h\to\infty}\min_{u\in H^{1}_{X,0}(\Omega)}\left(F_h^0(u)+G(u)\right)\\
&=\min_{u\in H^{1}_{X,0}(\Omega)}\left(F^0(u)+G(u)\right)=F^0(u_\infty)+G(u_\infty)\,.
\end{align*}
\medskip

{\bf 2nd step. } For a.e. $x\in\Om$ for any $\eta\in\mathbb{R}^m$ and for any $h\in\mathbb{N}$, let
\[
f_h(x,\eta):=\,\langle a^h(x)\eta,\eta\rangle\quad\text{and}\quad f(x,\eta):=\,\langle a(x)\eta,\eta\rangle\,.
\]
It is easy to see that $f_h$ and $f$ satisfy assumptions (i), (ii) and (iii) of Theorem \ref{convmom}. Moreover,
\[
\partial_\Phi \mathcal F_h(Xu_h)=\,a_h\,Xu_h\quad\text{and}\quad\partial_\Phi\mathcal F(Xu)=\,a\,Xu\,.
\]
Therefore, by the first step of the proof and by Corollary \ref{corconvmom}, we get the convergence of momenta.
\end{proof}


\end{document}